\documentclass{amsart}
\usepackage[T1]{fontenc}
\usepackage{mathtools}
\usepackage{xcolor}
\usepackage{amsmath,amssymb}
\usepackage{amsthm}
\usepackage{thmtools}
\usepackage{etoolbox}
\usepackage{tikz}
\usepackage{url}
\usepackage[backend=biber,style=alphabetic, maxbibnames=99, maxalphanames=6, url=false, doi=false, eprint=false]{biblatex}
\usepackage{geometry}
\usepackage{hyperref}
\hypersetup{
  hypertexnames = false, 
  colorlinks = true,
  linkcolor = blue,
  citecolor = magenta,
  pdftitle={Grothendieck Weights on Permutohedral Varieties and Matroids},
  pdfauthor={Yiyu Wang}
}
\usepackage{cleveref}

\usepackage{newpxtext}

\title{Grothendieck Weights on Permutohedral Varieties and Matroids}
\author{Yiyu Wang}
\address{Department of Mathematics, The Ohio State University, 231 W. 18th Ave., Columbus, OH 43210}
\email{wang.20315@osu.edu}
\date{Sep 6, 2026}

\declaretheorem[name=Theorem,numberwithin=section]{theorem}
\declaretheorem[name=Lemma,sibling=theorem]{lemma}
\declaretheorem[name=Proposition,sibling=theorem]{proposition}
\declaretheorem[name=Corollary,sibling=theorem]{corollary}
\declaretheorem[name=Remark,sibling=theorem,style=remark]{remark}
\declaretheorem[name=Definition,sibling=theorem,style=definition]{definition}
\declaretheorem[name=Example,sibling=theorem,style=definition]{example}

\makeatletter
\expandafter\patchcmd\csname\string\thmt@restatable\endcsname
  {\thmt@trivialref{thmt@@#3}{??}}
  {\protect\ref{thmt@@#3}}
  {}{\PackageError{manuscript}{Could not link restated theorem numbers}{}}
\makeatother

\crefname{theorem}{Theorem}{Theorems}
\Crefname{theorem}{Theorem}{Theorems}
\crefname{lemma}{Lemma}{Lemmas}
\Crefname{lemma}{Lemma}{Lemmas}
\crefname{proposition}{Proposition}{Propositions}
\Crefname{proposition}{Proposition}{Propositions}
\crefname{corollary}{Corollary}{Corollaries}
\Crefname{corollary}{Corollary}{Corollaries}
\crefname{definition}{Definition}{Definitions}
\Crefname{definition}{Definition}{Definitions}
\crefname{example}{Example}{Examples}
\Crefname{example}{Example}{Examples}
\crefname{remark}{Remark}{Remarks}
\Crefname{remark}{Remark}{Remarks}
\crefname{section}{Section}{Sections}
\Crefname{section}{Section}{Sections}
\crefname{subsection}{Subsection}{Subsections}
\Crefname{subsection}{Subsection}{Subsections}
\crefname{equation}{Equation}{Equations}
\Crefname{equation}{Equation}{Equations}

\AddToHook{env/theorem/begin}{\crefalias{section}{theorem}}
\AddToHook{env/lemma/begin}{\crefalias{section}{lemma}}
\AddToHook{env/proposition/begin}{\crefalias{section}{proposition}}
\AddToHook{env/corollary/begin}{\crefalias{section}{corollary}}
\AddToHook{env/definition/begin}{\crefalias{section}{definition}}
\AddToHook{env/example/begin}{\crefalias{section}{example}}
\AddToHook{env/remark/begin}{\crefalias{section}{remark}}

\DeclareMathOperator{\Hom}{Hom}
\DeclareMathOperator{\GW}{GW}
\DeclareMathOperator{\rk}{rk}
\DeclareMathOperator{\rank}{rank}
\DeclareMathOperator{\gr}{gr}
\DeclareMathOperator{\MC}{MC}

\DeclareMathOperator{\ch}{ch}

\DeclareMathOperator{\relint}{Relint}
\DeclareMathOperator{\Star}{Star}
\DeclareMathOperator{\cone}{cone}
\DeclareMathOperator{\face}{face}
\newcommand{\sHom}{\mathcal{H}om}
\newcommand{\sref}{\succneqq} 
\newcommand{\C}{{\mathbb{C}}}
\newcommand{\R}{{\mathbb{R}}}
\newcommand{\Q}{{\mathbb{Q}}}
\newcommand{\Z}{{\mathbb{Z}}}
\newcommand{\PP}{{\mathbb{P}}}

\begin{document}

\begin{abstract}
  Grothendieck weights, introduced by Shah, are \(K\)-theoretic analogues of Minkowski weights on smooth toric varieties. We study Grothendieck weights on the permutohedral fan and prove two main results: a \emph{\(K\)-balancing condition} that characterizes Grothendieck weights by a finite system of linear equations, and an explicit \emph{product rule} for the ring structure. We apply this framework to matroids, giving a combinatorial characterization of Grothendieck weights on matroidal fans. As the main application, we compute the motivic Chern class of the hyperplane arrangement complement in its wonderful compactification and show that the result depends only on the matroid and not on the realization. This allows us to extend the definition of the motivic Chern class to all loopless matroids.
\end{abstract}

\maketitle

\setcounter{tocdepth}{1}
\tableofcontents
\section{Introduction}

%

Let \(X\) be a complete smooth toric variety with fan \(\Sigma\).
In \cite{FultonSturmfels1997}, Fulton and Sturmfels introduce the notion of \emph{Minkowski weights} and show that the Chow ring of a complete smooth toric variety is isomorphic to the ring of Minkowski weights. A Minkowski weight is a \(\Z\)-valued function on \(\Sigma\) satisfying a \emph{balancing condition}. They also give a ``fan displacement rule'' for computing the product of two Minkowski weights.

In \cite{Shah2022}, Shah introduces the \(K\)-theoretic analogue of this theory. For a cone \(\sigma\in\Sigma\), let \(x_\sigma = [\mathcal{O}_{V(\sigma)}]\in K(X)\) denote the \(K\)-class of the structure sheaf of the orbit closure \(V(\sigma)\). A \emph{Grothendieck weight} is a \(\Z\)-valued function \(g:\Sigma\to\Z\) such that \(\sum_\sigma c_\sigma g(\sigma)=0\) whenever \(\sum_\sigma c_\sigma x_\sigma=0\). The \(K\)-ring of \(X\) is isomorphic to the group \(\GW(\Sigma)\) of all Grothendieck weights via the map \(\alpha\mapsto g_\alpha\), where \(g_\alpha(\sigma)=\chi(\alpha\cdot x_\sigma)\) and \(\chi:K(X)\to\Z\) is the Euler characteristic; this map is in fact a ring isomorphism. Shah also proves a product rule for computing the product of two Grothendieck weights. Both results are obtained by applying the Riemann--Roch transform to lift the corresponding results for Minkowski weights to \(K\)-theory.

However, Shah's approach has two drawbacks. Because it proceeds via the Riemann--Roch transform, which is an isomorphism only over \(\Q\), the resulting coefficients are rational rather than integral. Moreover, these coefficients are given as constant terms of formal power series involving exponential functions, making them difficult to compute explicitly in practice.


In this paper, we derive explicit, purely combinatorial formulas for the \(K\)-balancing condition and the product rule on the permutohedral toric variety.

We briefly recall the combinatorics of the permutohedral fan.
Write \([n] = \{1,\ldots,n\}\).
The \emph{permutohedral fan} \(\Sigma_{[n]}\) is the normal fan of the permutohedron \(\Pi_n\);
its cones are indexed by flags of nonempty proper subsets
\(\mathcal{F}\colon \emptyset \subsetneq F_1 \subsetneq \cdots \subsetneq F_k \subsetneq [n]\).
We write \(\ell(\mathcal{F}) = k\) for the \emph{length} of \(\mathcal{F}\), and
\(\mathcal{F} \sref \mathcal{G}\) to mean that \(\mathcal{F}\) \emph{strictly refines} \(\mathcal{G}\),
i.e., every set of \(\mathcal{G}\) appears in \(\mathcal{F}\) and \(\mathcal{F}\ne\mathcal{G}\).

A flag \(\mathcal{G}\) is \emph{\(\{i,j\}\)-neutral} if for each \(G\in\mathcal{G}\),
either \(\{i,j\}\subseteq G\) or \(\{i,j\}\cap G=\emptyset\)
(i.e., \(i\) and \(j\) are never separated by a set of \(\mathcal{G}\)).

\begin{restatable}{theorem}{permutohedralBalancing}\label{thm:k-balancing-permutohedron}
    A function \(g:\Sigma_{[n]}\to \Z\) is a Grothendieck weight if and only if for each pair \(i\ne j\) and each \(\{i,j\}\)-neutral flag \(\mathcal{G}\),
    \[
        \sum_{\mathcal{F}\in \mathbf{S}_{ij}(\mathcal{G})}(-1)^{\ell(\mathcal{F})}g(\mathcal{F}) = \sum_{\mathcal{F}\in \mathbf{S}_{ji}(\mathcal{G})}(-1)^{\ell(\mathcal{F})}g(\mathcal{F}),
    \]
    where \(\mathbf{S}_{ij}(\mathcal{G})\) denotes the set of strict refinements \(\mathcal{F}\succneqq\mathcal{G}\) in which every new set \(H\in\mathcal{F}\setminus\mathcal{G}\) contains \(i\) but not \(j\).
\end{restatable}


To state the product rule, let \(N = \Z^n/\Z(1,\ldots,1)\) be the cocharacter lattice of \(X_{[n]}\), and write \(N_\R = N\otimes_\Z \R\).

\begin{restatable}{theorem}{permutohedralProduct}\label{cor:product-rule-general}
    Let \(g_1, g_2\) be Grothendieck weights on \(X_{[n]}\), let \(v\in N\) be a generic vector, and let \(\gamma\in\Sigma_{[n]}\). Then
    \[
        (g_1\cdot g_2)(\gamma)
        = \sum_{\substack{\sigma\cap\tau=\gamma\\(\sigma+v)\cap \tau\neq\emptyset}}
        (-1)^{\dim \sigma + \dim \tau - n + 1 + \dim\gamma}\,
        g_1(\sigma)\,g_2(\tau).
    \]
    Here the sum ranges over cones \(\sigma,\tau\in\Sigma_{[n]}\). The sign uses the fact that \(\dim \Sigma_{[n]}=n-1\).
    The condition \(\sigma\cap\tau=\gamma\) is purely combinatorial and independent of \(v\).
\end{restatable}

\begin{remark}
    Let \((v_1,\dots,v_n)\) represent \(v\in N_\R\). Then \(v\) is generic for \(\Sigma_{[n]}\) if
    \[
        \sum_{i\in S}v_i\neq\sum_{i\in T}v_i
        \qquad\text{for all distinct }S,T\subseteq[n]\text{ with }|S|=|T|.
    \]
    See \cref{rmk:generic-permutohedron} for a proof. A convenient explicit choice is \(v_i=2^i\).
\end{remark}

In fact, both results hold in greater generality. We call a fan \(\Sigma\) \emph{strongly unimodular} (\cref{def:strongly-unimodular}) if it is unimodular and, for any two cones \(\sigma, \tau \in \Sigma\), the sublattice \(N_\sigma + N_\tau\) has index \(1\) or \(\infty\) in \(N\), where \(N_\sigma\) denotes the sublattice spanned by the rays of \(\sigma\). The permutohedral fan \(\Sigma_{[n]}\) is strongly unimodular, and this class is closed under taking subfans, star fans, and products. The \(K\)-balancing condition (\cref{thm:k-balancing}) and the product rule (\cref{thm:product-rule-arbitrary-cone}) both hold for any complete strongly unimodular fan.


We next apply the theory of Grothendieck weights to matroids.
For a loopless matroid \(\mathsf{M}\) on \(E = [n]\), Larson, Li, Payne, and Proudfoot \cite{LARSON2024109554} define the matroid \(K\)-ring \(K(\mathsf{M})\) and the matroidal fan \(\Sigma_{\mathsf{M}}\), whose cones are indexed by flags of nonempty proper flats of \(\mathsf{M}\).
As a direct application of \cref{thm:k-balancing-permutohedron}, we obtain a combinatorial characterization of Grothendieck weights on \(\Sigma_{\mathsf{M}}\).

\begin{restatable}{proposition}{matroidBalancing}\label{prop:K-balancing-for-matroid}
    A function \(g : \Sigma_{\mathsf{M}} \to \Z\) is a Grothendieck weight on \(\Sigma_{\mathsf{M}}\) if and only if for each pair \(i \ne j\) and each \(\{i,j\}\)-neutral flag of flats \(\mathcal{G}\) of \(\mathsf{M}\),
    \[
        \sum_{\mathcal{F}\in\mathbf{S}_{ij}(\mathcal{G})}(-1)^{\ell(\mathcal{F})} g(\mathcal{F})
        = \sum_{\mathcal{F}\in\mathbf{S}_{ji}(\mathcal{G})}(-1)^{\ell(\mathcal{F})} g(\mathcal{F}),
    \]
    where \(\mathbf{S}_{ij}(\mathcal{G})\) denotes the set of strict refinements \(\mathcal{F}\sref\mathcal{G}\) in which every new flat \(H\in\mathcal{F}\setminus\mathcal{G}\) contains \(i\) but not \(j\).
\end{restatable}

The main application of this paper is to the \emph{motivic Chern class} of a matroid.
For a realizable loopless matroid \(\mathsf{M}\) with realization \(L \subset \mathbb{C}^E\), let \(U_L = \mathbb{P}(L \cap (\mathbb{C}^*)^E)\) be the hyperplane arrangement complement, and let \(W_L = \overline{U_L} \subset X_{[n]}\) be its wonderful compactification.
The motivic Chern class \(\MC_y(\mathsf{M}) = \MC_y(U_L \to W_L) \in K(W_L)[y]\) is defined geometrically following \cite{BrasseletSchurmannYokura2010}.
Rather than computing the Grothendieck weight of \(\MC_y(\mathsf{M})\) directly, we work with its image under the normalized Serre duality \(\mathbb{D}\) on \(W_L\). Let \(g_{\mathsf{M}}^{\mathbb{D}}\) be the Grothendieck weight of \(\mathbb{D}(\MC_y(\mathsf{M}))\). This yields a cleaner combinatorial formula and suffices to recover \(\MC_y(\mathsf{M})\) itself.
For a matroid invariant \(\phi\) and a flag \(\mathcal{F} : \emptyset \subsetneq F_1 \subsetneq \cdots \subsetneq F_k \subsetneq E\), write
\[
    \phi(\mathsf{M})[\mathcal{F}] \coloneqq \phi(\mathsf{M}|F_1)\,\phi(\mathsf{M}|F_2/F_1)\cdots\phi(\mathsf{M}/F_k)
\]
for the product of \(\phi\) over the successive minors of \(\mathcal{F}\).

\begin{restatable}{proposition}{motivicChernWeight}\label{prop:GW-dual-MCy-matroid}
    Let \(\mathsf{M}\) be a realizable loopless matroid with realization \(L\), and let \(g_{\mathsf{M}}^{\mathbb{D}}\) denote the Grothendieck weight of \(\mathbb{D}(\MC_y(\mathsf{M})) \in K(W_L)[y]\), where \(\mathbb{D}_{W_L}(\mathcal{E}) = R\mathcal{H}om(\mathcal{E},\,\omega_{W_L}[\dim W_L])\) is the normalized Grothendieck--Serre duality on \(W_L\). For each flag of flats \(\mathcal{F}\),
    \[
        g_{\mathsf{M}}^{\mathbb{D}}(\mathcal{F}) = \frac{\chi_{\mathsf{M}}(-y)[\mathcal{F}]}{-1-y},
    \]
    where \(\chi_{\mathsf{M}}\) is the characteristic polynomial of \(\mathsf{M}\). In particular, the right-hand side is independent of the choice of realization \(L\).
\end{restatable}

By the \(K\)-balancing condition, this purely combinatorial formula defines a Grothendieck weight for every loopless matroid, including nonrealizable ones.

\begin{restatable}{theorem}{motivicChernGeneral}\label{thm:motivic-Chern-non-realizable}
    For any loopless matroid \(\mathsf{M}\), the function \(g_{\mathsf{M}}^{\mathbb{D}} : \Sigma_{\mathsf{M}} \to \Z[y]\) defined by \(g_{\mathsf{M}}^{\mathbb{D}}(\mathcal{F}) = \chi_{\mathsf{M}}(-y)[\mathcal{F}]/(-1-y)\) is a Grothendieck weight on \(\Sigma_{\mathsf{M}}\).
\end{restatable}

This allows us to define \(\MC_y(\mathsf{M})\) for all loopless matroids via Grothendieck--Serre duality (\cref{def:motivic-Chern-general}).
In the realizable case, applying the Hirzebruch class transformation \(T_{y*} = td_{(1+y)} \circ \MC_y\) and specializing at \(y = -1\) yields the Chern--Schwartz--MacPherson classes of matroids, recovering a result of López de Medrano, Rincón, and Shaw \cite{LopezdeMedranoRinconShaw2020}.

\begin{restatable}{corollary}{csmSpecialization}\label{cor:MCy-specializes-to-CSM}
    For a realizable loopless matroid \(\mathsf{M}\) and a flag of flats \(\mathcal{F}\) of length \(k\),
    \[
        \operatorname{csm}_k(\mathsf{M})(\mathcal{F}) = (-1)^{r-1-k}\,\beta(\mathsf{M})[\mathcal{F}]
    \]
    where \(r = \operatorname{rk}(\mathsf{M})\) and \(\beta\) denotes Crapo's beta invariant.
\end{restatable}


As a further application, we compute the Grothendieck weights of the tautological bundles associated with a matroid.
Following \cite{BergetEurSpinkTseng2023}, for a loopless matroid \(\mathsf{M}\) of rank \(r\) on \(E = [n]\), let \(\mathcal{S}_{\mathsf{M}}\) and \(\mathcal{Q}_{\mathsf{M}}\) be the tautological subbundle and quotient bundle on \(X_{[n]}\).

For a vector bundle \(\mathcal E\), we write
\[
    \lambda_t[\mathcal E]=\sum_{i\ge 0}[\wedge^i\mathcal E]t^i\in K(X)[t],
\]
and use the same notation for \(K\)-classes via the usual \(\lambda\)-ring structure.

\begin{restatable}{proposition}{tautologicalWeight}\label{prop:gw-tautological}
    The Grothendieck weight of \(\lambda_u(\mathcal{S}_{\mathsf{M}}^\vee)\lambda_v(\mathcal{Q}_{\mathsf{M}}^\vee)\) is the map
    \[
        \mathcal{F} \mapsto u^r T_{\mathsf{M}}\!\left(1+\tfrac{1}{u},\,1+v\right)[\mathcal{F}],
    \]
    where \(T_{\mathsf{M}}\) is the Tutte polynomial of \(\mathsf{M}\).
\end{restatable}

The Grothendieck weight of \(\lambda_u(\mathcal{S}_{\mathsf{M}}^\vee)\,\lambda_v(\mathcal{Q}_{\mathsf{M}}^\vee)\) can also be computed by applying the product rule to the individual Grothendieck weights of \(\lambda_u(\mathcal{S}_{\mathsf{M}}^\vee)\) and \(\lambda_v(\mathcal{Q}_{\mathsf{M}}^\vee)\). This yields the following identity.
Write \(I_{\mathsf{M}}(u) = \sum_{I \subseteq E\;\text{independent}} u^{|I|}\) for the independence polynomial of \(\mathsf{M}\), and let \(\operatorname{loop}(\mathsf{M})\) denote the number of loops of \(\mathsf{M}\).

\begin{restatable}{corollary}{tutteProduct}\label{cor:tutte-product-identity}
    For any loopless matroid \(\mathsf{M}\) of rank \(r\) on \(E = [n]\) and a generic vector \(w\in N\),
    \[
        u^r T_{\mathsf{M}}\!\left(1+\tfrac{1}{u},\,1+v\right)
        = \sum_{\substack{\sigma_{\mathcal{F}}\cap\sigma_{\mathcal{G}}=\{0\}\\ (\sigma_{\mathcal{F}}+w)\cap\sigma_{\mathcal{G}}\neq\emptyset}}
        (-1)^{\ell(\mathcal{F})+\ell(\mathcal{G})-n+1}\,
        I_{\mathsf{M}}(u)[\mathcal{F}]\,(1+v)^{\operatorname{loop}(\mathsf{M})}[\mathcal{G}],
    \]
    where \(\mathcal{F},\mathcal{G}\) range over all flags of nonempty proper subsets of \([n]\).
\end{restatable}

\subsection*{Acknowledgements}
The author thanks Eric Katz for his mentorship and guidance on this project, and Connor Simpson and Matt Larson for helpful comments on an earlier draft.


\subsection*{Outline of the paper}
Section~2 recalls background on the \(K\)-ring of smooth toric varieties, Grothendieck weights, and the permutohedral fan.
Section~3 introduces strongly unimodular fans and establishes their basic properties.
Section~4 proves the \(K\)-balancing condition (\cref{thm:k-balancing}) for complete strongly unimodular fans.
Section~5 derives the product rule (\cref{cor:product-rule-general}).
Section~6 specializes these results to the permutohedral fan \(\Sigma_{[n]}\) and computes the Grothendieck weights of generalized permutohedra.
Section~7 extends the framework to matroids, giving a combinatorial characterization of Grothendieck weights on matroidal fans.
Section~8 defines and studies the motivic Chern class of a matroid.
Section~9 computes the Grothendieck weights of tautological bundles on the permutohedral variety and derives a Tutte polynomial identity.

\section{Preliminaries}
In this section, we collect some preliminary results on Grothendieck weights and permutohedra. Most of the results are well-known, but we include them here for completeness and to set up the notation.

\subsection{The Grothendieck ring of a smooth toric variety}\label{subsec:toric-k-ring}
Let \(\Sigma\) be a unimodular fan in a real vector space \(N_\R\) with lattice \(N\). Let \(M=\Hom(N,\Z)\) be the dual lattice, and let \(T=\operatorname{Spec} \C[M]\) be the torus. Let \(X_\Sigma\) be the toric variety associated with the fan \(\Sigma\), with dense torus \(T\). For a cone \(\sigma\in \Sigma\), we use \(V(\sigma)\) to denote the closure of the \(T\)-orbit \(O(\sigma)\) that corresponds to \(\sigma\).

For an algebraic variety \(X\), one can define its \(K\)-ring \(K(X)\) as the Grothendieck group of locally free sheaves of finite rank on \(X\), with multiplication given by tensor product. If every coherent sheaf on \(X\) admits a finite locally free resolution, then the natural map \(K(X)\to G_0(X)\) to the Grothendieck group of coherent sheaves is an isomorphism. This applies to the smooth toric varieties considered below, and we use this identification throughout. The \(K\)-ring of \(X_\Sigma\) can be described by the following theorem, which is a restatement of \cite[Theorem~1]{CCKR2025}.

\begin{theorem}\label{thm:K_ring_toric}
    Let \(\Sigma\) be a unimodular fan in \(N_\R\), and let \(X_\Sigma\) be the associated smooth toric variety. Then
    \[
        K(X_\Sigma)\cong
        \frac{\mathbb{Z}[x_\rho\mid \rho\in\Sigma(1)]}
        {I_{SR}+I_{\mathrm{lin}}},
    \]
    where
    \[
        I_{\mathrm{SR}}=\left\langle\prod_{\rho\in S}x_\rho\ \middle|\ S\subseteq\Sigma(1),\ \operatorname{cone}(S)\notin\Sigma\right\rangle
    \]
    and
    \[
        I_{\mathrm{lin}}=\left\langle
        \prod_{\langle m,u_\rho\rangle>0}(1-x_\rho)^{\langle m,u_\rho\rangle}
        -\prod_{\langle m,u_\rho\rangle<0}(1-x_\rho)^{-\langle m,u_\rho\rangle}
        \ \middle|\ m\in M
        \right\rangle.
    \]
    Under this isomorphism, each variable \(x_\rho\) maps to \([\mathcal{O}_{V(\rho)}]\in K(X_\Sigma)\).
\end{theorem}

Extending the notation, for a cone \(\sigma\in\Sigma\), we denote the class of the structure sheaf \(\mathcal{O}_{V(\sigma)}\) by
\[
    x_\sigma=\prod_{\rho\in\sigma(1)} x_\rho.
\]
In particular, \(x_\emptyset=1\).

Furthermore, if \(\Sigma\) is complete, there is a natural map \(\chi:K(X_\Sigma)\to\Z\) taking a coherent sheaf \(\mathcal{F}\) to its Euler characteristic \(\chi(\mathcal{F})\). It is well known that \(\chi(x_\sigma)=1\) for all \(\sigma\in\Sigma\). In fact, \(\chi(\mathcal{O}_X)=1\) for every smooth rational proper variety \(X\); see, for example, \cite[Theorem~9.2.3]{CoxLittleSchenck2011}.

Another important result is that \(\chi\) can be used to define a perfect pairing on \(K(X)\).

\begin{theorem}[{\cite[Theorems~1.2 and~1.3]{AndersonPayne2015}}]\label{thm:euler-pairing-perfect}
    Suppose \(\Sigma\) is complete and unimodular, and let \(X_\Sigma\) be the associated smooth complete toric variety. Then
    \[
        \langle a,b\rangle:=\chi(ab),\qquad a,b\in K(X_\Sigma),
    \]
    defines a perfect pairing
    \[
        K(X_\Sigma)\otimes_{\mathbb{Z}}K(X_\Sigma)\longrightarrow \mathbb{Z}.
    \]
\end{theorem}

\subsection{Grothendieck weights}
Following \cite[Definition~2.4]{Shah2022}, we use the following definition.
\begin{definition}[Grothendieck weight]\label{def:gw}
    Let \(\Sigma\) be a complete unimodular fan, and let \(X_\Sigma\) be the associated toric variety. A \emph{Grothendieck weight} on \(\Sigma\) is a function \(g:\Sigma\to\mathbb{Z}\) such that for every relation
    \[
        \sum_{\sigma\in\Sigma} c_\sigma[\mathcal{O}_{V(\sigma)}]=0
        \quad\text{in } K(X_\Sigma),
    \]
    one has
    \[
        \sum_{\sigma\in\Sigma} c_\sigma g(\sigma)=0.
    \]
    The abelian group of all Grothendieck weights on \(\Sigma\) is denoted by \(\GW(\Sigma)\).
\end{definition}

\begin{proposition}\label{prop:K-bijection-GW}
    Let \(\Sigma\) be a complete and unimodular fan. The map
    \[
        \Phi:K(X_\Sigma)\to \GW(\Sigma),\qquad
        \Phi(\alpha)(\sigma):=\chi(\alpha\cdot x_\sigma)
    \]
    is an isomorphism of abelian groups.
\end{proposition}
\begin{proof}
    If \(\Phi(\alpha)=0\), then \(\chi(\alpha\cdot x_\sigma)=0\) for every \(\sigma\in\Sigma\). Since the classes \(x_\sigma\) span \(K(X_\Sigma)\) by \cref{thm:K_ring_toric}, the perfectness of the Euler pairing in \cref{thm:euler-pairing-perfect} implies \(\alpha=0\). Thus \(\Phi\) is injective.

    Now let \(g\in \GW(\Sigma)\). Since \(g\) vanishes on all linear relations among the classes \(x_\sigma\), the assignment
    \[
        \lambda_g\Bigl(\sum_\sigma c_\sigma x_\sigma\Bigr)\coloneq \sum_\sigma c_\sigma g(\sigma)
    \]
    defines a well-defined homomorphism \(\lambda_g:K(X_\Sigma)\to \Z\). By \cref{thm:euler-pairing-perfect}, there exists \(\alpha\in K(X_\Sigma)\) such that
    \[
        \lambda_g(\beta)=\chi(\alpha\cdot\beta)
    \]
    for all \(\beta\in K(X_\Sigma)\). In particular,
    \[
        g(\sigma)=\lambda_g(x_\sigma)=\chi(\alpha\cdot x_\sigma)=\Phi(\alpha)(\sigma)
    \]
    for every \(\sigma\in\Sigma\). Hence \(\Phi\) is surjective.
\end{proof}

We endow \(\GW(\Sigma)\) with the unique ring structure for which \(\Phi\) is a ring isomorphism.

For \(\alpha\in K(X_\Sigma)\), write \(g_\alpha:=\Phi(\alpha)\). Equivalently, \(g_\alpha:\Sigma\to\Z\) is given by
\[
    g_\alpha(\sigma)=\chi(\alpha\cdot x_\sigma),\qquad \sigma\in\Sigma.
\]

In particular, \(g_\alpha(\emptyset)=\chi(\alpha)\).

\subsection{The permutohedron and the permutohedral toric variety}\label{subsec:permutohedron}
Let \(n\ge 2\), and write \([n]=\{1,2,\dots,n\}\). The (classical) permutohedron is
\[
    \Pi_n:=\operatorname{conv}\bigl\{(\pi(1),\pi(2),\dots,\pi(n))\in\mathbb{R}^n:\pi\in\mathfrak{S}_n\bigr\},
\]
see \cite{Postnikov2009}. It lies in the affine hyperplane
\[
    x_1+x_2+\cdots+x_n=\frac{n(n+1)}2.
\]
We use the lattice
\[
    N=\mathbb{Z}^n/\mathbb{Z}(1,\dots,1),
    \qquad
    M=\Hom(N,\mathbb{Z})\cong\{(a_1,\dots,a_n)\in\mathbb{Z}^n:\textstyle\sum_i a_i=0\}.
\]
For every nonempty proper subset \(S\subsetneq [n]\), let
\[
    \bar{e}_S\in N
\]
be the image of \(\sum_{i\in S}e_i\) in \(N\), where \(\bar e_1,\dots,\bar e_n\) are the standard basis vectors in \(\mathbb{Z}^n\).

A flag of proper nonempty subsets is a chain
\[
    \mathcal{F}:\emptyset\subsetneq F_1\subsetneq F_2\subsetneq\cdots\subsetneq F_k\subsetneq [n].
\]
We define
\[
    \sigma_{\mathcal{F}}:=\operatorname{cone}(\bar e_{F_1},\bar e_{F_2},\dots,\bar e_{F_k})\subset N_{\mathbb{R}}.
\]
If \(\pi=S_1|S_2|\cdots|S_r\) is an ordered set partition of \([n]\), its associated flag is
\[
    \mathcal{F}_\pi:\emptyset\subsetneq S_1\subsetneq S_1\cup S_2\subsetneq\cdots\subsetneq S_1\cup\cdots\cup S_{r-1}\subsetneq [n].
\]

\begin{proposition}[cf.~\cite{Postnikov2009}]
    The cones \(\sigma_{\mathcal{F}}\) form a complete unimodular fan \(\Sigma_{[n]}\), and \(\Sigma_{[n]}\) is the normal fan of \(\Pi_n\). Equivalently, cones of \(\Sigma_{[n]}\) are indexed by ordered set partitions of \([n]\), via \(\pi\leftrightarrow\mathcal{F}_\pi\).
\end{proposition}

The toric variety \(X_{[n]}\) associated to \(\Sigma_{[n]}\) is called the \emph{permutohedral variety}. Since \(\Sigma_{[n]}\) is complete and unimodular (hence smooth) and also polytopal (as a normal fan), \(X_{[n]}\) is a smooth projective toric variety of dimension \(n-1\); see \cite{Postnikov2009,CoxLittleSchenck2011}.
The nef torus-invariant divisors on \(X_{[n]}\) are in one-to-one correspondence with generalized permutohedra.

\begin{definition}
    A lattice polytope \(P\subset M_\R\) is called a \emph{generalized permutohedron} if
    its normal fan \(\mathcal{N}(P)\) coarsens \(\Sigma_{[n]}\).
\end{definition}

A generalized permutohedron \(P\) defines a torus-invariant divisor \(D_P\) as follows:
\[
    D_P=\sum_{\emptyset\subsetneq S\subsetneq E}-\min_{m\in P}\langle m, \bar{e}_S \rangle [V(\rho_S)],
\]
where \(\rho_S\) denotes the ray \(\cone(\bar{e}_S)\) in the fan \(\Sigma_{[n]}\).

\begin{proposition}[see, e.g., {\cite[Section 2.7]{BergetEurSpinkTseng2023}}]
    A lattice polytope \(P\) determines a nef torus-invariant divisor \(D_P\) if and only if \(P\) is a generalized permutohedron.
\end{proposition}

\section{Strongly unimodular fans}
Although the main focus of this paper is the permutohedron, our arguments work for the larger class defined below.

\begin{definition}\label{def:strongly-unimodular}
    A fan \(\Sigma\subseteq N_\R\) is called \emph{strongly unimodular} if it is unimodular and, for any two cones \(\sigma,\tau\in\Sigma\),
    \[
        [N:N_\sigma+N_\tau]=1 \text{ or }\infty.
    \]
    Here unimodular means that each cone is generated by part of a \(\Z\)-basis of \(N\).
    Thus, in addition, if \((N_\sigma)_\R\) and \((N_\tau)_\R\) generate the whole \(N_\R\), then the index of \(N_\sigma+N_\tau\) is one.
\end{definition}

We collect some useful properties of strongly unimodular fans in the following proposition.

\begin{proposition}\label{proposition:strongly-unimodular-properties}
    Let \(\Sigma\subseteq N_\R\) be strongly unimodular.
    \begin{enumerate}
        \item Every subfan \(\Sigma'\subseteq \Sigma\) is strongly unimodular.
        \item For each \(\sigma\in\Sigma\), the star fan of \(\sigma\) is strongly unimodular.
        \item If two fans \(\Sigma_1\) and \(\Sigma_2\) are strongly unimodular, then so is \(\Sigma_1\times \Sigma_2\).
    \end{enumerate}
\end{proposition}
\begin{proof}
    \begin{enumerate}
        \item This is immediate from the definition: subfans of unimodular fans are unimodular, and any pair of cones in \(\Sigma'\) is also a pair of cones in \(\Sigma\).

        \item Fix \(\sigma\in\Sigma\). Since \(\Sigma\) is unimodular, \(N_\sigma\) is saturated in \(N\), and the quotient \(N/N_\sigma\) is a lattice. The star fan is unimodular because any cone \(\alpha\supseteq\sigma\) is generated by part of a basis of \(N\), and its image in \(N/N_\sigma\) is generated by the complementary basis vectors.

              It remains to check the index condition. Let \(\overline{\alpha},\overline{\beta}\in\Star_\Sigma(\sigma)\), represented by cones \(\alpha,\beta\in\Sigma\) containing \(\sigma\). In the quotient lattice \(N/N_\sigma\), we have
              \[
                  N_{\overline{\alpha}}=\frac{N_\alpha}{N_\sigma},\qquad N_{\overline{\beta}}=\frac{N_\beta}{N_\sigma}.
              \]
              Thus
              \[
                  \frac{N/N_\sigma}{N_{\overline{\alpha}}+N_{\overline{\beta}}}
                  \cong
                  \frac{N}{N_\alpha+N_\beta},
              \]
              and therefore
              \[
                  [N/N_\sigma:N_{\overline{\alpha}}+N_{\overline{\beta}}]
                  =
                  [N:N_\alpha+N_\beta].
              \]
              The right-hand side is \(1\) or \(\infty\) by strong unimodularity of \(\Sigma\), so \(\Star_\Sigma(\sigma)\) is strongly unimodular.

        \item Let \(\Sigma_i\subseteq (N_i)_\R\) (\(i=1,2\)) be strongly unimodular. The product fan is unimodular because a product of cones generated by parts of bases is generated by part of the product basis. It remains to check the index condition. Take cones
              \[
                  \sigma=\sigma_1\times \sigma_2,\qquad \tau=\tau_1\times \tau_2
              \]
              in \(\Sigma_1\times\Sigma_2\). Then
              \[
                  N_\sigma=N_{\sigma_1}\oplus N_{\sigma_2},\qquad
                  N_\tau=N_{\tau_1}\oplus N_{\tau_2},
              \]
              so
              \[
                  N_\sigma+N_\tau=(N_{\sigma_1}+N_{\tau_1})\oplus(N_{\sigma_2}+N_{\tau_2}).
              \]
              Hence
              \[
                  [N_1\oplus N_2:N_\sigma+N_\tau]
                  =
                  [N_1:N_{\sigma_1}+N_{\tau_1}]\cdot [N_2:N_{\sigma_2}+N_{\tau_2}],
              \]
              with the convention that if one factor is \(\infty\), then the product is \(\infty\). Since each factor is \(1\) or \(\infty\), the product is \(1\) or \(\infty\). Therefore \(\Sigma_1\times\Sigma_2\) is strongly unimodular.
    \end{enumerate}
\end{proof}

\begin{remark}\label{rmk:non_unimodular_example}
    The index condition alone does not imply unimodularity. Let \(N=\Z^4\), with standard basis \(e_1,e_2,e_3,e_4\), and set
    \[
        \sigma=\operatorname{cone}(u_1,u_2),\qquad
        u_1=e_1,\ \ u_2=e_1+2e_2.
    \]
    Let \(\Sigma\) be the fan consisting of \(\sigma\), its faces, and the two rays \(\R_{\ge0}e_3\) and \(\R_{\ge0}e_4\). For any pair \(\alpha,\beta\in\Sigma\), the vector space \(\operatorname{span}_\R(\alpha)+\operatorname{span}_\R(\beta)\) has dimension at most \(3\). Hence
    \[
        [N:N_\alpha+N_\beta]=\infty.
    \]
    Thus \(\Sigma\) satisfies the index condition.

    On the other hand, \(\sigma\) is not unimodular: writing
    \[
        L:=N\cap \operatorname{span}_\R(\sigma)=\Z e_1\oplus \Z e_2,
    \]
    we have
    \[
        N_\sigma=\Z u_1+\Z u_2=\{(a+b,2b,0,0):a,b\in\Z\},
    \]
    so \([L:N_\sigma]=2\). Thus \(\Sigma\) is not unimodular.
\end{remark}

\begin{remark}
    The index condition alone is preserved under coarsening: if \(\Sigma\) refines \(\Sigma'\), one chooses subcones of a pair of cones in \(\Sigma'\) with the same spans and applies the index condition in \(\Sigma\). This does not imply that coarsenings are strongly unimodular: a coarsening of a unimodular fan need not be simplicial, hence need not be unimodular.
\end{remark}

\begin{example}
    The fan of \(\PP^n\) is strongly unimodular. Let \(N=\Z^n\), let \(e_1,\dots,e_n\) be the standard basis, and set
    \[
        e_0=-\sum_{i=1}^n e_i,
    \]
    so that \(\Sigma(1)=\{e_0,e_1,\dots,e_n\}\). If \([N:N_\sigma+N_\tau]<\infty\), then \(N_\sigma+N_\tau\) has rank \(n\), so it contains \(n\) linearly independent rays from \(\Sigma(1)\). Any \(n\)-subset of \(\Sigma(1)\) is an integral basis of \(N\), hence \([N:N_\sigma+N_\tau]=1\). Therefore the fan of \(\PP^n\) is strongly unimodular.
\end{example}

\begin{example}\label{eg:braid_is_strongly_unimodular}
    The braid fan \(\Sigma_{[n]}\) is strongly unimodular. Since the braid fan is unimodular, it remains to show that \([N:N_\sigma+N_\tau]\in\{1,\infty\}\) for each pair of cones \((\sigma,\tau)\).

    If \(\dim(N_\sigma+N_\tau)<n-1\), then \([N:N_\sigma+N_\tau]=\infty\). Thus we may assume
    \[
        \dim(N_\sigma+N_\tau)=n-1.
    \]
    Let \(r=\dim\sigma\) and \(s=\dim\tau\). Then \(r+s\ge n-1\). We first reduce to the case \(r+s=n-1\). If \(r+s>n-1\), choose a basis \(B\) of \(N_\R\) from \(\sigma(1)\cup\tau(1)\), and define
    \[
        \sigma'=\operatorname{cone}(B\cap\sigma(1)),\qquad
        \tau'=\operatorname{cone}(B\setminus\sigma(1)).
    \]
    Since \(B\subseteq \sigma(1)\cup\tau(1)\), we have \(B\setminus\sigma(1)\subseteq \tau(1)\). Hence \(N_{\sigma'}+N_{\tau'}\subseteq N_\sigma+N_\tau\), and
    \[
        \dim\sigma'+\dim\tau'=|B|=n-1.
    \]
    Therefore, it is enough to prove \([N:N_{\sigma'}+N_{\tau'}]=1\), and we may assume \(r+s=n-1\).

    Suppose \(\sigma=\sigma_\pi\), where \(\pi=S_1|S_2|\cdots|S_{r+1}\) is an ordered partition of \([n]\). The ray generators of \(N_\sigma\) are (brackets denote the class in \(N=\Z^n/\Z(1,\ldots,1)\))
    \[
        [e_{S_1}], [e_{S_1\cup S_2}], \ldots, [e_{S_1\cup\cdots\cup S_{r}}],
    \]
    and similarly, if \(\tau=\tau_{\pi'}\) with \(\pi'=T_1|T_2|\cdots|T_{s+1}\), the ray generators of \(N_\tau\) are
    \[
        [e_{T_1}], [e_{T_1\cup T_2}], \ldots, [e_{T_1\cup\cdots\cup T_{s}}],
    \]
    It suffices to show that these vectors contain an integral basis of \(N=\Z^n/\Z(1,\ldots,1)\). Equivalently, after adding \(e_{[n]}\), they contain a basis of \(\Z^n\). Let \(M\) be the \(n\times n\) integer matrix whose rows are the above vectors, with the last row \(e_{[n]}\).

    We now apply unimodular row operations. Replacing \(R_i\) by \(R_i-R_{i-1}\) for \(i=r,r-1,\ldots,2\), the first \(r\) rows become
    \[
        e_{S_1},e_{S_2},\dots, e_{S_{r}}.
    \]
    Applying the same operation to the \(T\)-block gives \(e_{T_1},\dots,e_{T_s}\), and then replacing the last row by
    \[
        e_{[n]}-\sum_{j=1}^s e_{T_j}=e_{T_{s+1}},
    \]
    we obtain a new matrix \(M'\). Since \(\dim(N_\sigma+N_\tau)=n-1\), the matrix \(M\) has rank \(n\), hence \(M'\) is also nonsingular.

    Consider the bipartite graph \(G\) with left vertices \(S_1,\dots,S_{r+1}\), right vertices \(T_1,\dots,T_{s+1}\), and one edge for each \(k\in[n]\), joining the unique pair \((S_i,T_j)\) with \(k\in S_i\cap T_j\). The matrix \(M'\) is obtained by deleting the row indexed by \(S_{r+1}\) from the vertex-edge incidence matrix of \(G\), so \(M'\) is a minor of that incidence matrix. Incidence matrices of bipartite graphs are totally unimodular; see the appendix theorem in \cite{HellerTompkins1956}. Therefore \(M'\) is totally unimodular. Since \(M'\) is nonsingular, \(\det M'=\pm1\). Hence its rows form a \(\Z\)-basis of \(\Z^n\), which implies \([N:N_\sigma+N_\tau]=1\).
\end{example}

\section{\texorpdfstring{$K$}{K}-balancing conditions on strongly unimodular fans}

Throughout this section, we assume that $\Sigma$ is a complete strongly unimodular fan of dimension $n$. Recall that the $K$-ring of $X=X_\Sigma$ is spanned by $x_\sigma=[\mathcal{O}_{V(\sigma)}]$, and the balancing condition for Grothendieck weights is given by the linear relations between these elements.

We start with
\[
    R_m=\prod_{\langle m,u_\rho\rangle>0}(1-x_\rho)^{\langle m,u_\rho\rangle}
    -\prod_{\langle m,u_\rho\rangle<0}(1-x_\rho)^{-\langle m,u_\rho\rangle}.
\]

We claim that if we choose $m$ carefully, all exponents \(\langle m,u_\rho\rangle\in\{0,\pm1\}\).

\begin{lemma}\label{lem:pm1-dual-basis}
    Suppose \(\Sigma\) is a complete strongly unimodular fan of dimension \(n\), and define
    \[
        Q_\Sigma:=\{m\in M_\R:-1\leq \langle m,u_\rho\rangle\leq 1\text{ for all }\rho\in\Sigma(1)\}.
    \]
    Then:
    \begin{enumerate}
        \item \(Q_\Sigma\cap M\) contains an integral basis of \(M\).
        \item More generally, for every \(\tau\in\Sigma\), \(Q_\Sigma\cap M(\tau)\) contains an integral basis of \(M(\tau):=\tau^\perp\cap M\).
    \end{enumerate}
\end{lemma}
\begin{proof}
    We only prove (2), since (1) is the special case \(\tau=\{0\}\).
    Since \(\Sigma\) is strongly unimodular, it is unimodular by definition.
    Fix \(\tau\in\Sigma\), and write \(r=\dim\tau\). Choose a maximal cone \(\gamma\in\Sigma(n)\) containing \(\tau\). After reindexing rays, we may write
    \[
        \tau=\operatorname{cone}(u_1,\dots,u_r),\qquad
        \gamma=\operatorname{cone}(u_1,\dots,u_n),
    \]
    with \(u_1,\dots,u_n\) a \(\Z\)-basis of \(N\). Let \(v_1,\dots,v_n\) be the dual basis of \(M\), i.e.
    \[
        \langle v_i,u_j\rangle=\delta_{ij}, \qquad \forall i,j \in [n].
    \]

    Fix any ray \(\rho\in\Sigma(1)\), and write
    \[
        u_\rho=\sum_{i=1}^n a_i u_i,\qquad a_i=\langle v_i,u_\rho\rangle\in\Z
    \]
    for the primitive generator of \(\rho\). We claim that \(a_i\in\{0,\pm 1\}\).

    For each \(i\), let
    \[
        \sigma_i:=\operatorname{cone}(u_1,\dots,\widehat{u_i},\dots,u_n)\in\Sigma.
    \]
    Then \(N_{\sigma_i}=\bigoplus_{j\ne i}\Z u_j\), \(N_\rho=\Z u_\rho\), and strong unimodularity gives
    \[
        [N:N_{\sigma_i}+N_\rho]\in\{1,\infty\}.
    \]
    If \(a_i=0\), then \(u_\rho\in (N_{\sigma_i})_\R\), so the index is \(\infty\). If \(a_i\neq 0\), then \(N_{\sigma_i}+N_\rho\) has rank \(n\), and
    \[
        [N:N_{\sigma_i}+N_\rho]=|a_i|,
    \]
    since modulo \(N_{\sigma_i}\), the class of \(u_\rho\) equals \(a_i[u_i]\). Hence \(|a_i|=1\). Therefore \(a_i\in\{0,\pm1\}\) for all \(i\), i.e.
    \[
        \langle v_i,u_\rho\rangle\in\{0,\pm1\}\quad\text{for all }i,\rho.
    \]
    In particular, each \(v_i\in Q_\Sigma\). Now observe that \(v_{r+1},\dots,v_n\) annihilate \(N_\tau\), hence lie in \(M(\tau)\), and they form an integral basis of \(M(\tau)\). Therefore
    \[
        \{v_{r+1},\dots,v_n\}\subseteq Q_\Sigma\cap M(\tau),
    \]
    so \(Q_\Sigma\cap M(\tau)\) contains an integral basis of \(M(\tau)\).
\end{proof}

We omit $\Sigma$ and use $Q$ to denote the polytope $Q_\Sigma$ when the fan is clear from context.

For each \(q\in Q\cap M\), let
\[
    \mathbf{P}_q=\{\rho\in\Sigma(1):\langle q,u_\rho\rangle=1\},\qquad
    \mathbf{N}_q=\{\rho\in\Sigma(1):\langle q,u_\rho\rangle=-1\}.
\]
Since \(q\in Q\cap M\), we have \(\langle q,u_\rho\rangle\in\{-1,0,1\}\) for every ray \(\rho\). Thus \(R_q\) can be rewritten as
\[
    R_q=\prod_{\rho\in \mathbf{P}_q}(1-x_\rho)-\prod_{\rho\in \mathbf{N}_q}(1-x_\rho).
\]
Expanding the products, we get
\[
    R_q=\sum_{\substack{\sigma\in \Sigma\\\sigma(1)\subseteq \mathbf{P}_q}}(-1)^{\dim \sigma}x_\sigma -
    \sum_{\substack{\sigma\in \Sigma\\\sigma(1)\subseteq \mathbf{N}_q}}(-1)^{\dim \sigma}x_\sigma.
\]
This is a linear relation among the \(x_\sigma\)'s. To obtain all linear relations, fix a cone \(\tau\) such that \(q\) annihilates \(N_\tau\), i.e. \(\langle q,u_\rho\rangle=0\) for all \(\rho\in\tau(1)\). Then no ray of \(\tau\) lies in \(\mathbf{P}_q\cup \mathbf{N}_q\), so \(\tau(1)\cap\sigma(1)=\emptyset\) for every cone \(\sigma\) appearing in the sums above.

When two cones \(\tau,\sigma\) do not share a common ray, we have
\[x_\tau x_\sigma=\begin{cases}
    x_{\cone(\tau,\sigma)},& \text{ if }\cone(\tau,\sigma)\in \Sigma,\\
    0,& \text{ if }\cone(\tau,\sigma) \text{ is not a cone of } \Sigma.
\end{cases}\]
Here and throughout this paper, \(\cone(\sigma,\tau)\) denotes the minimal cone in $\Sigma$ that contains both \(\sigma,\tau\).

Therefore,
\[
    x_\tau R_q = \sum_{\substack{\sigma(1)\subseteq \mathbf{P}_q\\\cone(\sigma,\tau)\in\Sigma}}(-1)^{\dim \sigma}x_{\cone(\sigma,\tau)} -
    \sum_{\substack{\sigma(1)\subseteq \mathbf{N}_q\\\cone(\sigma,\tau)\in\Sigma}}(-1)^{\dim \sigma}x_{\cone(\sigma,\tau)},
\]
where both sums range over cones \(\sigma\in\Sigma\).

We may rewrite this linear relation in another way. Instead of summing over $\sigma\in\Sigma$ such that $\cone(\sigma,\tau)$ exists in \(\Sigma\), we let $\sigma$ range over all cones that contain $\tau$ and such that \(\sigma(1)\setminus \tau(1)\) is a subset of $\mathbf{P}_q$ or $\mathbf{N}_q$. The two sets are in one-to-one correspondence by $\alpha\to\cone(\alpha,\tau)$.

\begin{equation}\label{eq:linear-relation-among-x-sigma}
    (-1)^{\dim\tau}x_\tau R_q = \sum_{\substack{\sigma\supsetneq \tau\\\sigma(1)\setminus\tau(1)\subseteq \mathbf{P}_q}}(-1)^{\dim \sigma}x_{\sigma} -
    \sum_{\substack{\sigma\supsetneq \tau\\\sigma(1)\setminus\tau(1)\subseteq \mathbf{N}_q}}(-1)^{\dim \sigma}x_{\sigma}.
\end{equation}

If $q$ annihilates \(N_\tau\), we call $\tau$ a $q$-neutral cone, or \((q,\tau)\) an admissible pair. Each admissible pair \((q,\tau)\) gives a linear relation among the \(x_\sigma\)'s.
We claim that these relations generate all linear relations among the \(x_\sigma\)'s. 

\begin{theorem}\label{thm:k-balancing}
    Let \(\Sigma\) be a complete strongly unimodular fan, and let \(X=X_\Sigma\). A function \(g:\Sigma\to\Z\) is a Grothendieck weight if and only if for each pair \((q,\tau)\) with \(q\in Q\cap M\) and \(\tau\) a \(q\)-neutral cone, one has
    \[
        \sum_{\substack{\sigma\supsetneq \tau\\\sigma(1)\setminus\tau(1)\subseteq \mathbf{P}_q}}(-1)^{\dim \sigma}g(\sigma) =
        \sum_{\substack{\sigma\supsetneq \tau\\\sigma(1)\setminus\tau(1)\subseteq \mathbf{N}_q}}(-1)^{\dim \sigma}g(\sigma).
    \]
\end{theorem}
\begin{proof}
    Let $\Psi\colon\Q^\Sigma \to K(X)\otimes_\Z \Q$ be the linear map $(c_\sigma)_{\sigma\in\Sigma}\mapsto \sum_{\sigma}c_\sigma x_\sigma$. By definition, $g\colon\Sigma\to\Z$ is a Grothendieck weight if and only if $g$ annihilates every element of $\ker\Psi$.

    For each pair $(q,\tau)$ with $q\in Q\cap M$ and $\tau$ a $q$-neutral cone, define $b_{(q,\tau)}\in\Q^\Sigma$ by
    \[
        (b_{(q,\tau)})_\sigma=\begin{cases}
            (-1)^{\dim \sigma},   & \text{if } \tau\subsetneq\sigma,\ \sigma(1)\setminus \tau(1)\subseteq \mathbf{P}_q, \\
            (-1)^{\dim \sigma+1}, & \text{if } \tau\subsetneq\sigma,\ \sigma(1)\setminus \tau(1)\subseteq \mathbf{N}_q, \\
            0,                    & \text{otherwise}.
        \end{cases}
    \]
    Since $\tau$ is $q$-neutral, $q\in M(\tau)=\tau^\perp\cap M$, so no ray of $\tau$ lies in $\mathbf{P}_q\cup \mathbf{N}_q$. Note that the vector \(b_{(q,\tau)}\) is exactly the coefficients in \cref{eq:linear-relation-among-x-sigma}. Therefore $b_{(q,\tau)}\in\ker\Psi$, which proves the \emph{only if} direction: every Grothendieck weight $g$ satisfies $\langle g,b_{(q,\tau)}\rangle=0$ for all admissible pairs $(q,\tau)$.

    For the \emph{if} direction, let $M_{\mathrm{GW}}$ be the matrix with row vectors $b_{(q,\tau)}$, rows indexed by admissible pairs $(q,\tau)$ and columns by $\sigma\in\Sigma$. Since all rows lie in $\ker\Psi$,
    \[
        \rank(M_{\mathrm{GW}})\leq\dim \ker\Psi=|\Sigma|-\dim K(X)\otimes_\Z \Q.
    \]
    It suffices to establish the reverse inequality
    \[
        \rank(M_{\mathrm{GW}})\geq|\Sigma|-\dim K(X)\otimes_\Z \Q.
    \]

    The strategy is to compare $M_{\mathrm{GW}}$ with a Minkowski-weight balancing matrix, using \(\gr K(X)\otimes_\Z\Q\cong A^*(X)\otimes_\Z\Q\). Let
    \[
        \Psi'\colon\Q^\Sigma \to A^*(X)\otimes_\Z \Q,\qquad
        (c_\sigma)_{\sigma\in\Sigma}\longmapsto \sum_{\sigma\in\Sigma}c_\sigma [V(\sigma)].
    \]
    Here $A^k(X)$ denotes the Chow group of codimension-$k$ cycles on $X$, and $A^*(X)=\bigoplus_{k=0}^n A^k(X)$ is the graded Chow ring.

    For \(k=0,1,\dots,n\), write \(\Sigma^{(k)}:=\{\sigma\in\Sigma:\operatorname{codim}\sigma=k\}\).
    Fix \(k\in\{0,\dots,n-1\}\), \(\tau\in\Sigma^{(k+1)}\), and \(q\in Q\cap M(\tau)\). In the terminology of Fulton--Sturmfels \cite[Section~2]{FultonSturmfels1997}, a Minkowski $k$-weight \(c\colon\Sigma^{(k)}\to\Q\) satisfies the codimension-one balancing condition along $\tau$ in direction $m\in M(\tau)$:
    \[
        \sum_{\substack{\sigma\in\Sigma^{(k)}\\ \sigma\supset\tau}} c(\sigma)\langle m, u_{\sigma/\tau}\rangle = 0,
    \]
    where $u_{\sigma/\tau}$ denotes the primitive generator of the unique ray of $\sigma$ not in $\tau(1)$. Since $q\in M(\tau)\cap Q$, \cref{lem:pm1-dual-basis} gives $\langle q,u_{\sigma/\tau}\rangle\in\{-1,0,1\}$. Grouping terms by the sign of this pairing, the condition for $m=q$ becomes
    \[
        \sum_{\substack{\sigma\supset\tau\\ \sigma(1)\setminus\tau(1)\subseteq \mathbf{P}_q}} c(\sigma)
        -
        \sum_{\substack{\sigma\supset\tau\\ \sigma(1)\setminus\tau(1)\subseteq \mathbf{N}_q}} c(\sigma)=0.
    \]
    Here and in the next display, both sums range over cones \(\sigma\in\Sigma^{(k)}\).
    By \cref{lem:pm1-dual-basis}, $Q\cap M(\tau)$ contains an integral basis of $M(\tau)$; hence letting $q$ range over $Q\cap M(\tau)$ yields all codimension-one balancing conditions along $\tau$. Since \(\dim\sigma=n-k\) is constant on \(\Sigma^{(k)}\), multiplying through by $(-1)^{n-k}$ gives the equivalent form
    \[
        \sum_{\substack{\sigma\supset\tau\\ \sigma(1)\setminus\tau(1)\subseteq \mathbf{P}_q}} (-1)^{\dim\sigma} c(\sigma)
        +
        \sum_{\substack{\sigma\supset\tau\\ \sigma(1)\setminus\tau(1)\subseteq \mathbf{N}_q}} (-1)^{\dim\sigma+1} c(\sigma)=0.
    \]
    Define a row vector \(b^{\mathrm{MW}}_{(q,\tau)}\in\Q^{\Sigma^{(k)}}\) by 
    \[
        \bigl(b^{\mathrm{MW}}_{(q,\tau)}\bigr)_\sigma=
        \begin{cases}
            (-1)^{\dim\sigma},   & \text{if }\sigma\supset\tau,\ \sigma(1)\setminus\tau(1)\subseteq \mathbf{P}_q, \\
            (-1)^{\dim\sigma+1}, & \text{if }\sigma\supset\tau,\ \sigma(1)\setminus\tau(1)\subseteq \mathbf{N}_q, \\
            0,                   & \text{otherwise},
        \end{cases}
    \]
    where \(\sigma\) ranges over all \(\Sigma^{(k)}\). Note that this description is identical to the row vector \(b_{(q,\tau)}\), except \(\sigma\) only ranges over codimension \(k\) cones.
    Let \(M_{\mathrm{MW}}^{(k)}\) be the matrix with row vectors \(b^{\mathrm{MW}}_{(q,\tau)}\), rows indexed by \((q,\tau)\) with \(\tau\in\Sigma^{(k+1)}\), \(q\in Q\cap M(\tau)\), and columns indexed by \(\Sigma^{(k)}\).

    By the preceding paragraph and the Fulton--Sturmfels description of Minkowski weights \cite[Theorem~2.1]{FultonSturmfels1997}, the row span of \(M_{\mathrm{MW}}^{(k)}\) is the space of linear relations among the classes $[V(\sigma)]$, $\sigma\in\Sigma^{(k)}$, in $A^{n-k}(X)$. Therefore
    \[
        \rank\bigl(M_{\mathrm{MW}}^{(k)}\bigr) = \dim\ker\bigl(\Psi'|_{\Q^{\Sigma^{(k)}}}\bigr).
    \]

    Now form one large block diagonal matrix \(M_{\mathrm{MW}}\), using each \(M_{\mathrm{MW}}^{(k)}\) as a diagonal block. Since the row vectors of the two matrices \(M_{\mathrm{GW}}\) and \(M_{\mathrm{MW}}\) have the same description except for the range of \(\sigma\), we conclude that \(M_\mathrm{GW}\) is block upper-triangular with block entries \(M_\mathrm{MW}^{(k)}\).

    Since \([V(\sigma)]\in A^{n-k}(X)\) for each \(\sigma\in\Sigma^{(k)}\), the map \(\Psi'\) respects the codimension grading. The zero cone contributes no relation, so the preceding rank computation gives
    \[
        \rank(M_{\mathrm{MW}})
        =
        \sum_{k=0}^{n-1}\rank\!\bigl(M_{\mathrm{MW}}^{(k)}\bigr)
        =
        \dim\ker\Psi'
        =
        |\Sigma|-\dim A^*(X)\otimes_\Z\Q.
    \]

    Since \(M_{\mathrm{GW}}\) is block upper-triangular, and the corresponding block diagonal matrix is \(M_{\mathrm{MW}}\),
    \[
        \rank(M_{\mathrm{GW}})\ge \rank(M_{\mathrm{MW}})=|\Sigma|-\dim A^*(X)\otimes_\Z\Q.
    \]

    Since \(\gr K(X)\otimes_\Z\Q\cong A^*(X)\otimes_\Z\Q\) (see \cite[Example~15.2.16]{Fulton1998}), we have \(\dim K(X)\otimes_\Z\Q=\dim A^*(X)\otimes_\Z\Q\), so
    \[
        \rank(M_{\mathrm{GW}})\ge |\Sigma|-\dim K(X)\otimes_\Z\Q.
    \]
    Together with the reverse inequality, $\rank(M_{\mathrm{GW}})=\dim\ker\Psi$, so the vectors $\{b_{(q,\tau)}\}$ span $\ker\Psi$. Hence any function $g$ annihilating all $b_{(q,\tau)}$ is a Grothendieck weight, completing the proof.
\end{proof}

\begin{remark}\label{rmk:smaller-Q}
    The index set $Q\cap M$ in \cref{thm:k-balancing} can be replaced by any subset $S\subseteq Q\cap M$ such that for every $\tau\in\Sigma$, the set $S\cap M(\tau)$ contains an integral basis of $M(\tau)$. Indeed, the proof of \cref{thm:k-balancing} only uses this property of $Q\cap M$.
\end{remark}

\section{Product rule for Grothendieck weights}

In this section, we describe the ring structure of Grothendieck weights on a strongly unimodular fan. Throughout this section, we assume \(\Sigma\) is a complete strongly unimodular fan of dimension \(n\) in \(N_\R\), and \(X=X_\Sigma\) is the corresponding toric variety.

We first state the goal of this section. Given two Grothendieck weights \(g_1,g_2\), view them as elements of \(K(X)\). The product \(g_1\cdot g_2\) is completely determined by \(g_1\) and \(g_2\), and our goal is to find coefficients \(m_{\sigma,\tau}^\gamma\) such that
\begin{equation}\label{eq:product-coeff}
    (g_1\cdot g_2)(\gamma)=\sum_{\sigma,\tau}m_{\sigma,\tau}^\gamma\, g_1(\sigma)\,g_2(\tau).
\end{equation}

Following \cite{FultonSturmfels1997}, we first reduce the problem to the study of the diagonal of toric varieties.

\begin{proposition}
    Let \(\delta\colon X\to X\times X\) be the diagonal embedding. Any coefficients \(m_{\sigma,\tau}^\gamma\) in an expansion of \(\delta_*[\mathcal{O}_{V(\gamma)}]\) give product coefficients. More precisely, if
    \[
        \delta_*[\mathcal{O}_{V(\gamma)}]
        = \sum_{\sigma,\tau} m_{\sigma,\tau}^\gamma\,
        [\mathcal{O}_{V(\sigma\times\tau)}]
    \]
    in the ring \(K(X\times X)\), then \(m_{\sigma,\tau}^\gamma\) can be used in \cref{eq:product-coeff}.
\end{proposition}
\begin{proof}
    Let
    \[
        g_1\boxtimes g_2 \coloneqq p_1^*g_1\cdot p_2^*g_2 \in K(X\times X),
    \]
    where \(p_1,p_2\colon X\times X\to X\) are the projections. Multiply both sides by \(g_1\boxtimes g_2\) and then apply \(\chi\).

    On the left-hand side, by the projection formula,
    \[
        \chi\bigl((g_1\boxtimes g_2)\cdot \delta_*[\mathcal{O}_{V(\gamma)}]\bigr)
        = \chi\bigl(\delta^*(g_1\boxtimes g_2)\cdot [\mathcal{O}_{V(\gamma)}]\bigr).
    \]
    By the definition of \(\delta^*\), we have \(g_1\cdot g_2=\delta^*(g_1\boxtimes g_2)\), so the left-hand side equals
    \[
        (g_1\cdot g_2)(\gamma)
        = \chi\bigl((g_1\boxtimes g_2)\cdot \delta_*[\mathcal{O}_{V(\gamma)}]\bigr).
    \]

    On the right-hand side, since \([\mathcal{O}_{V(\sigma\times\tau)}]=[\mathcal{O}_{V(\sigma)}]\boxtimes[\mathcal{O}_{V(\tau)}]\),
    \[
        \chi\bigl((g_1\boxtimes g_2)\,[\mathcal{O}_{V(\sigma\times\tau)}]\bigr)
        = \chi\bigl(g_1[\mathcal{O}_{V(\sigma)}]\bigr)\,
        \chi\bigl(g_2[\mathcal{O}_{V(\tau)}]\bigr)
        = g_1(\sigma)\,g_2(\tau).
    \]
    Comparing the two sides finishes the proof.
\end{proof}

Since the classes \([\mathcal{O}_{V(\sigma)}]\) are not linearly independent (the balancing condition furnishes the linear relations), such coefficients \(m_{\sigma,\tau}^\gamma\) are not unique. We describe one way to find a solution.

We first treat the case \(\gamma=\{0\}\). The general case reduces to this one by replacing \(\Sigma\) with \(\operatorname{Star}(\gamma)\). We need to express \(\delta_*[\mathcal{O}_{X}]=[\mathcal{O}_{\delta(X)}]\) as a linear combination of classes of \(T\times T\)-invariant subvarieties. The most natural way to do this is via one-parameter subgroup degeneration of \(\delta(X)\).

Following Fulton--Sturmfels \cite[Theorem~3.2]{FultonSturmfels1997}, we choose \(v\in N\) generic in the sense of the next lemma.

\begin{lemma}\label{lem:generic-v-choice}
    There exists a nonempty dense open subset \(U\subset N_\R\) such that for any \(v\in U\), the following hold:
    \begin{enumerate}
        \item For every \(\sigma,\tau\in\Sigma\), if \((\sigma+v)\cap\tau\neq\emptyset\), then \((\sigma^\circ+v)\cap(\tau^\circ)\neq\emptyset\).
        \item If \((\sigma+v)\cap\tau\neq\emptyset\), then \[\dim ((\sigma+v)\cap\tau)=\dim\sigma+\dim\tau-\dim N_\R.\]
    \end{enumerate}
\end{lemma}
For a polyhedron \(C\), we use \(C^\circ\) to denote its relative interior, and \(\partial C=C\setminus C^\circ\) to denote its boundary.
\begin{proof}
    Write \(L_\sigma=(N_\sigma)_\R\) for the linear span of \(\sigma\), and set
    \[
        U\coloneqq N_\R\setminus
        \bigcup_{\substack{\sigma,\tau\in\Sigma\\L_\sigma+L_\tau\neq N_\R}}
        (L_\sigma+L_\tau).
    \]
    This is the complement of a finite union of proper linear subspaces, hence is nonempty, dense, and open.

    Fix \(v\in U\) and suppose \((\sigma+v)\cap\tau\neq\emptyset\). Then
    \[
        v\in C\coloneqq\tau-\sigma\subseteq L_\sigma+L_\tau,
    \]
    so \(L_\sigma+L_\tau=N_\R\), and \(C\) is full-dimensional.

    We claim \(v\in C^\circ\). If \(v\in\partial C\), choose a supporting functional \(0\neq u\in M_\R\) that is nonnegative on \(C\) and satisfies \(\langle u,v\rangle=0\). Then \(\sigma'=\sigma\cap u^\perp\) and \(\tau'=\tau\cap u^\perp\) are faces of \(\sigma\) and \(\tau\), and
    \[
        v\in\tau'-\sigma'\subseteq L_{\sigma'}+L_{\tau'}\subseteq u^\perp\subsetneq N_\R,
    \]
    contradicting \(v\in U\). Thus \(v\in C^\circ=\tau^\circ-\sigma^\circ\), proving \((\sigma^\circ+v)\cap\tau^\circ\neq\emptyset\).

    Since the relative interiors meet,
    \[
        \operatorname{aff}\bigl((\sigma+v)\cap\tau\bigr)=(L_\sigma+v)\cap L_\tau.
    \]
    Consequently,
    \[
        \dim\bigl((\sigma+v)\cap\tau\bigr)=\dim(L_\sigma\cap L_\tau)=\dim\sigma+\dim\tau-\dim N_\R.
    \]
\end{proof}

We fix such an integral vector \(v\in U\cap N\).

Let \(\lambda_v\colon\mathbb{G}_m \to T\) be the one-parameter subgroup defined by \(v\).  Consider the family
\[
    \Gamma \subset \mathbb{G}_m\times X\times X,
    \qquad
    \Gamma = \bigl\{(t,\,\lambda_{-v}(t)\,p,\,p) \bigm| t\in \mathbb{G}_m,\; p\in X\bigr\}.
\]

Consider the scheme-theoretic closure of \(\Gamma\) in \(\mathbb{P}^1\times X\times X\), and let
\[
    \Gamma_t \coloneqq \overline{\Gamma}\cap (\{t\}\times X\times X),
    \qquad t\in \mathbb{P}^1.
\]
The intersection is scheme-theoretic. Clearly \(\Gamma_1=\delta(X)\); we set \(Z\coloneqq\Gamma_0\).

The family \(\overline{\Gamma}\) is flat over \(\mathbb{P}^1\). Indeed, since \(\Gamma\) is smooth and connected, \(\overline{\Gamma}\) is integral. The projection \(\overline{\Gamma}\to \mathbb{P}^1\) is dominant and \(\overline{\Gamma}\) is integral, so it is flat. A standard consequence of flatness is
\[
    [\mathcal{O}_{\Gamma_0}] = [\mathcal{O}_{\Gamma_1}],
\]
which means
\[
    [\mathcal{O}_{\delta(X)}] = [\mathcal{O}_Z].
\]

We next identify the special fiber scheme-theoretically using local toric initial degenerations.  The following elementary index comparison is used in the normalized-volume computation.

\begin{lemma}\label{lem:index-duality}
    Let \(\sigma,\tau\in \Sigma\). Assume
    \[
        \dim \sigma+\dim \tau=n
        \quad\text{and}\quad
        \rank(N_\sigma+N_\tau)=n.
    \]
    Then
    \[
        [M:\sigma^\perp+\tau^\perp]=[N:N_\sigma+N_\tau].
    \]
\end{lemma}
\begin{proof}
    Put \(a=\dim\sigma\), \(b=\dim\tau\), so \(a+b=n\). Choose \(\mathbb{Z}\)-bases \(u_1,\dots,u_a\) of \(N_\sigma\) and \(v_1,\dots,v_b\) of \(N_\tau\), and let \(A\in\mathrm{Mat}_{n\times n}(\mathbb{Z})\) be the matrix with these vectors as columns (in a fixed basis of \(N\)). Since \(\mathrm{im}(A)=N_\sigma+N_\tau\) has rank \(n\),
    \[
        [N:N_\sigma+N_\tau]=|\det A|.
    \]

    Since \(N_\sigma\) is a saturated sublattice of \(N\), it is a direct summand, so \(M\to N_\sigma^\vee\cong\mathbb{Z}^a\), \(m\mapsto m|_{N_\sigma}\), is surjective with kernel \(\sigma^\perp\); similarly for \(\tau\). In the given bases, the map
    \[
        \phi\colon M\longrightarrow M/\sigma^\perp\oplus M/\tau^\perp\cong\mathbb{Z}^a\oplus\mathbb{Z}^b,
        \qquad
        \phi(m)=\bigl(m+\sigma^\perp,\;m+\tau^\perp\bigr)
    \]
    has matrix \(A^T\). Its kernel is \(\sigma^\perp\cap\tau^\perp=(N_\sigma+N_\tau)^\perp=0\), so \(|\mathrm{coker}\,\phi|=|\det A^T|=|\det A|\). The surjection \((\bar m_1,\bar m_2)\mapsto\overline{m_1-m_2}\) from \(M/\sigma^\perp\oplus M/\tau^\perp\) to \(M/(\sigma^\perp+\tau^\perp)\) has kernel \(\mathrm{im}\,\phi\), so
    \[
        [M:\sigma^\perp+\tau^\perp]=|\mathrm{coker}\,\phi|=|\det A|=[N:N_\sigma+N_\tau].\qedhere
    \]
\end{proof}

\begin{lemma}\label{lem:diagonal-degeneration-reduced}
    Assume \(\Sigma\) is complete and strongly unimodular.  For the family
    \[
        \Gamma=\{(t,\lambda_{-v}(t)p,p):t\in\mathbb G_m,\ p\in X\},
    \]
    with \(v\) chosen as above, the special fiber \(Z=\Gamma_0\) is reduced.  More
    precisely,
    \[
        Z=
        \bigcup_{\substack{(\sigma+v)\cap\tau\neq\emptyset\\
                \dim\sigma+\dim\tau=n}}
        V(\sigma)\times V(\tau)
    \]
    scheme-theoretically.
\end{lemma}

\begin{proof}
    It is enough to work on maximal affine charts \(U_{\sigma_0}\times U_{\tau_0}\).
    Since \(\Sigma\) is complete and strongly unimodular, it is smooth.  Let
    \(\{a_\rho:\rho\in\sigma_0(1)\}\) and
    \(\{b_{\rho'}:\rho'\in\tau_0(1)\}\) be the dual bases of \(M\) determined by the
    ray bases of \(\sigma_0\) and \(\tau_0\).  On this chart the diagonal is defined by
    the toric ideal
    \[
        I_{\sigma_0,\tau_0}
        =
        \ker\Bigl(
        \mathbb C[x_\rho,y_{\rho'}]\longrightarrow\mathbb C[M],\quad
        x_\rho\mapsto\chi^{a_\rho},\
        y_{\rho'}\mapsto\chi^{b_{\rho'}}
        \Bigr).
    \]
    The translate by \(\lambda_{-v}\) gives
    \[
        x_\rho\mapsto t^{-\langle a_\rho,v\rangle}\chi^{a_\rho},
        \qquad
        y_{\rho'}\mapsto\chi^{b_{\rho'}}.
    \]
    Set
    \[
        \eta_v(x_\rho)=\eta_v(a_\rho)=\langle a_\rho,v\rangle,
        \qquad
        \eta_v(y_{\rho'})=\eta_v(b_{\rho'})=0 .
    \]
    The graph closure over this chart is the Groebner degeneration of
    \(I_{\sigma_0,\tau_0}\) with respect to \(\eta_v\); hence its special fiber is
    defined by the minimum-weight initial ideal
    \(\operatorname{in}_{\eta_v}(I_{\sigma_0,\tau_0})\)
    \cite[Chapter~1]{Sturmfels1996}.

    The special fiber uses the minimum \(\eta_v\)-initial ideal.  In Sturmfels'
    convention for toric initial ideals, this corresponds to using the height
    vector \(-\eta_v\) on the configuration
    \[
        \mathcal{A}=
        \{a_\rho:\rho\in\sigma_0(1)\}
        \cup
        \{b_{\rho'}:\rho'\in\tau_0(1)\}
        \subset M .
    \]
    The resulting regular subdivision is obtained from the lower faces of the
    lifted configuration
    \[
        \{(a_\rho,-\eta_v(a_\rho)):\rho\in\sigma_0(1)\}
        \cup
        \{(b_{\rho'},-\eta_v(b_{\rho'})):\rho'\in\tau_0(1)\}
        \subset M_\mathbb R\oplus\mathbb R
    \]
    \cite[Theorem~8.3 and Chapter~8]{Sturmfels1996}.  In this convention, a subset
    of the configuration \(\mathcal{A}\) is a cell exactly when there is some
    \(p\in N_\mathbb R\) such that
    \[
        m\longmapsto \langle m,p\rangle-\eta_v(m)
    \]
    is zero on that subset and strictly positive on all other vectors.  We now
    determine the cells. For a lower cell defined by \(p\), the values
    on the two parts of \(\mathcal{A}\) are
    \[
        \langle a_\rho,p-v\rangle,\qquad
        \langle b_{\rho'},p\rangle .
    \]
    Since \(a_\rho\) and \(b_{\rho'}\) are the dual bases, the nonnegativity of these
    numbers gives \(p-v\in\sigma_0\) and \(p\in\tau_0\).  Let \(\sigma\le\sigma_0\)
    and \(\tau\le\tau_0\) be the unique faces with
    \[
        p-v\in\sigma^\circ,\qquad p\in\tau^\circ .
    \]
    Then the zero set of the function above is exactly
    \[
        \{a_\rho\in \sigma^\perp\}
        \cup
        \{b_{\rho'}:\rho'\in\tau^\perp\}.
    \]
    Conversely, any \(p\) satisfying these two relative-interior conditions defines
    this cell.  Thus the cells in this chart are precisely the displayed subsets,
    with \((\sigma^\circ+v)\cap\tau^\circ\neq\emptyset\).  By
    \cref{lem:generic-v-choice}, this is the same as
    \((\sigma+v)\cap\tau\neq\emptyset\).

    We claim each maximal cell of the subdivision above is a unimodular simplex.
    Under the correspondence above, inclusion of cells is opposite to componentwise
    inclusion of the pairs \((\sigma,\tau)\).  Hence a cell is maximal exactly when
    \((\sigma,\tau)\) is minimal among the pairs with
    \((\sigma+v)\cap\tau\neq\emptyset\).  In the common refinement, these minimal
    pairs are exactly the vertices, i.e. those for which \((\sigma+v)\cap\tau\) is a
    point.  Under the genericity condition in \cref{lem:generic-v-choice}, every
    nonempty intersection satisfies
    \[
        \dim((\sigma+v)\cap\tau)=\dim\sigma+\dim\tau-n.
    \]
    Thus the maximal cells are exactly those with
    \[
        (\sigma+v)\cap\tau\neq\emptyset,
        \qquad
        \dim\sigma+\dim\tau=n.
    \]

    Note that the vectors \(a_\rho\in \sigma^\perp\) form a lattice basis of
    \(\sigma^\perp\), and similarly the corresponding \(b_{\rho'}\)'s form a
    lattice basis of \(\tau^\perp\). By the genericity of \(v\), \(\rank(N_\sigma+N_\tau)=n\), hence \(\sigma^\perp\cap\tau^\perp= \{0\}\).  Each such maximal cell contains
    \[
        (n-\dim\sigma)+(n-\dim\tau)=n
    \]
    vectors, and these vectors span \(M_\mathbb R\); hence it is a simplex cell. Therefore such a maximal cell has normalized volume
    \[
        [M:\sigma^\perp+\tau^\perp]
        =
        [N:N_\sigma+N_\tau]
        =
        1.
    \]
    Here the second equality is \cref{lem:index-duality}, and the last equality is
    strong unimodularity. Hence every maximal cell is a unimodular simplex, so the regular subdivision is in fact a unimodular triangulation. A toric initial ideal is
    square-free if and only if the corresponding regular triangulation is unimodular
    \cite[Corollary~8.9]{Sturmfels1996}.  Therefore
    \(\operatorname{in}_{\eta_v}(I_{\sigma_0,\tau_0})\) is square-free monomial.
    It follows that \(Z\cap(U_{\sigma_0}\times U_{\tau_0})\) is reduced on each maximal affine
    chart.  These charts cover \(X\times X\), so \(Z\) is reduced.  The same
    maximal cells give the displayed orbit-closure decomposition.
\end{proof}

Consequently, in the notation above,
\[
    C = \bigl\{(\sigma,\tau) \bigm|
    (\sigma+v)\cap \tau\neq \emptyset,\;
    \dim \sigma + \dim \tau = n\bigr\}.
\]

To compute the class \([\mathcal{O}_Z]\in K(X\times X)\), we use the following inclusion-exclusion complex.

\begin{lemma}
    Let \(X\) be a smooth toric variety with fan \(\Sigma\) and \(Z_i=V(\sigma_i)\) for some \(\sigma_i\in \Sigma\), \(i=1,2,\ldots,r\). Then the following C\v{e}ch complex for the closed cover \(\{Z_i\}\)
    \[
        0 \to \mathcal{O}_{\bigcup_i Z_i}
        \to \bigoplus_i \mathcal{O}_{Z_i}
        \to \bigoplus_{i<j} \mathcal{O}_{Z_i\cap Z_j}
        \to \cdots
        \to \mathcal{O}_{\bigcap_i Z_i}
        \to 0
    \]
    is an exact sequence.
\end{lemma}
\begin{proof}
    The proof relies on two facts. First, the \emph{Mayer--Vietoris} exact sequence for two closed subschemes \(A,B\) of a scheme \(X\):
    \[
        0 \to \mathcal{O}_{A\cup B}
        \to \mathcal{O}_A \oplus \mathcal{O}_B
        \to \mathcal{O}_{A\cap B}
        \to 0.
    \]
    Second, for three \(Z_i=V(\sigma_i)\), we have
    \[
        Z_3 \cap (Z_1 \cup Z_2)
        = (Z_3\cap Z_1) \cup (Z_3\cap Z_2)
    \]
    scheme-theoretically. This holds because in each affine chart of \(X\), the ideals of the \(Z_i\) are generated by monomials. Monomial ideals satisfy the distributive law (see, e.g., \cite[Lemma~7.3.2]{MooreRogersSatherWagstaff2018}):
    \[
        I_3 + (I_1 \cap I_2) = (I_3 + I_1) \cap (I_3 + I_2).
    \]

    Using these two facts, the lemma follows by induction on~\(r\).
\end{proof}
\begin{remark}
    The above complex is not necessarily exact for general \(X\) and \(Z_i\) if \(r>2\).
\end{remark}

Let \(Z_i=V(\sigma_i)\times V(\tau_i)\) be the components of \(Z\), with \((\sigma_i,\tau_i)\in C\). Applying the lemma to \(Z=\bigcup_i Z_i\) gives
\[
    [\mathcal{O}_Z]
    = \sum_{\emptyset\neq S\subseteq C} (-1)^{|S|+1}\,
    [\mathcal{O}_{\bigcap_{i\in S} Z_i}].
\]

In smooth toric varieties, scheme-theoretic intersections of torus-invariant subvarieties are again torus-invariant:
\[
    V(\sigma)\cap V(\tau) = V(\operatorname{cone}(\sigma,\tau))
\]
if \(\operatorname{cone}(\sigma,\tau)=\sigma+\tau\in \Sigma\), and empty otherwise. The operation \(\operatorname{cone}(\cdot,\cdot)\) is the \emph{join} of two cones in the poset of cones of~\(\Sigma\). We reinterpret the above summation as follows.

Let
\[
    P_v = \bigl\{(\sigma,\tau) \bigm| (\sigma+v)\cap \tau\neq \emptyset\bigr\}
    \cup\{\hat{0},\hat{1}\},
\]
with the partial order defined by \((\sigma,\tau) \leq (\sigma',\tau')\) if \(\sigma\subseteq\sigma'\) and \(\tau\subseteq\tau'\), together with \(\hat{0} \leq (\sigma,\tau) \leq \hat{1}\) for all \((\sigma,\tau)\). We first show that
\begin{lemma}
    The poset \(P_v\) is a lattice. More precisely, for any two elements \((\sigma,\tau),(\sigma',\tau')\in P_v\), their meet and join are given by
    \begin{align*}
        (\sigma,\tau) \wedge (\sigma',\tau')
         & = \begin{cases}
                 (\sigma\cap \sigma', \tau\cap \tau'),
                  & \text{if } (\sigma\cap \sigma',\tau\cap \tau')\in P_v, \\
                 \hat{0},
                  & \text{otherwise},
             \end{cases}   \\[6pt]
        (\sigma,\tau) \vee (\sigma',\tau')
         & = \begin{cases}
                 \bigl(\operatorname{cone}(\sigma,\sigma'),\,
                 \operatorname{cone}(\tau,\tau')\bigr),
                  & \text{if } \operatorname{cone}(\sigma,\sigma')\in \Sigma
                 \text{ and } \operatorname{cone}(\tau,\tau')\in \Sigma,     \\
                 \hat{1},
                  & \text{otherwise}.
             \end{cases}
    \end{align*}
\end{lemma}
\begin{proof}
    We verify that the two operations are indeed the meet and join in \(P_v\). The meet is obvious. For the join: if \(\operatorname{cone}(\sigma,\sigma')\notin \Sigma\) or \(\operatorname{cone}(\tau,\tau')\notin \Sigma\), then there is no upper bound of \((\sigma,\tau)\) and \((\sigma',\tau')\) in \(P_v\) other than \(\hat{1}\). If both cones lie in \(\Sigma\), it suffices to show that \(\bigl(\operatorname{cone}(\sigma,\sigma'),\, \operatorname{cone}(\tau,\tau')\bigr)\in P_v\), i.e.\ that \(\bigl(\operatorname{cone}(\sigma,\sigma')+v\bigr)\cap \operatorname{cone}(\tau,\tau')\neq \emptyset\). This is clear, since
    \[
        (\sigma+v)\cap \tau
        \subset \bigl(\operatorname{cone}(\sigma,\sigma')+v\bigr)
        \cap \operatorname{cone}(\tau,\tau'),
    \]
    and the left-hand side is nonempty.
\end{proof}

In particular, every \((\sigma_i,\tau_i)\) belongs to \(P_v\), as do all their joins. Let \(L_v\) be the sublattice of \(P_v\) generated by all \((\sigma_i,\tau_i)\in C\). The set \(C\) consists of the atoms of \(L_v\); in particular, \(L_v\) is an atomic lattice. The inclusion-exclusion formula can be rewritten as
\[
    [\mathcal{O}_Z]
    = \sum_{(\sigma,\tau)\in L_v}
    -\mu_{L_v}\!\bigl(\hat{0},(\sigma,\tau)\bigr)\,
    [\mathcal{O}_{V(\sigma\times\tau)}].
\]
This is the classical crosscut theorem in poset theory \cite[Theorem~3]{Rota1964}.

To compute the Möbius function of \(L_v\), we show that \(P_v\) is isomorphic to the face poset of the polyhedral complex given by the common refinement of \(\Sigma\) and \(\Sigma+v\). This uses the genericity of \(v\).

\begin{lemma}
    Let \(\mathcal{C}=\mathcal{C}(\Sigma,v)\) be the polyhedral complex given by the common refinement of \(\Sigma\) and \(\Sigma+v\). The poset \(P_v\) is isomorphic to the face lattice \(F(\mathcal{C})\) of~\(\mathcal{C}\).
\end{lemma}
\begin{proof}
    Recall that the face lattice of a polyhedral complex is the set of all faces of all polyhedra, together with a minimum element \(\emptyset\) and a maximum element \(\hat{1}\), ordered by inclusion. We construct mutually inverse poset isomorphisms \(\Phi\colon P_v\to F(\mathcal{C})\) and \(\Psi\colon F(\mathcal{C})\to P_v\).

    For any \((\sigma,\tau)\in P_v\), define
    \[
        \Phi\bigl((\sigma,\tau)\bigr) = (\sigma+v)\cap \tau.
    \]
    Set \(\Phi(\hat{0})=\emptyset\) and \(\Phi(\hat{1})=\hat{1}\). Conversely, for any face \(F\in F(\mathcal{C})\), choose any point \(x\) in its interior. Let \(\sigma(F)\in \Sigma\) and \(\tau(F)\in \Sigma\) be the minimal cones such that \(x\in \sigma(F)+v\) and \(x\in \tau(F)\). Define
    \[
        \Psi(F) = \bigl(\sigma(F),\,\tau(F)\bigr).
    \]

    We verify:
    \begin{enumerate}
        \item \emph{\(\Phi\) and \(\Psi\) are well-defined.}
              For \(\Phi\), this is clear. For \(\Psi\), the choice of \(x\) does not matter because \(\relint(F)=(\sigma^\circ+v)\cap \tau^\circ\) for some \(\sigma,\tau\in \Sigma\); any interior point of \(F\) determines the same minimal cones.

        \item \emph{\(\Phi\circ \Psi =\mathrm{Id}_{F(\mathcal{C})}\).}
              For any face \(F\in F(\mathcal{C})\), we need
              \[
                  \Phi(\Psi(F)) = \bigl(\sigma(F)+v\bigr)\cap \tau(F) = F.
              \]
              This follows because \(\relint(F)=(\sigma(F)^\circ+v)\cap \tau(F)^\circ=\relint\bigl((\sigma(F)+v)\cap \tau(F)\bigr)\) by \cref{lem:generic-v-choice}.

        \item \emph{\(\Psi\circ \Phi = \mathrm{Id}_{P_v}\).}
              For any \((\sigma,\tau)\in P_v\), choose \(x\in \relint\bigl((\sigma+v)\cap \tau\bigr)\). The minimal cones containing \(x\) are exactly \(\sigma\) and \(\tau\), again by \(\relint\bigl((\sigma+v)\cap \tau\bigr)=(\sigma^\circ+v)\cap \tau^\circ\).

        \item \emph{\(\Phi\) and \(\Psi\) are order-preserving.}
              If \((\sigma,\tau)\leq (\sigma',\tau')\) in \(P_v\), then \(\sigma\subseteq \sigma'\) and \(\tau\subseteq \tau'\), whence
              \[
                  (\sigma+v)\cap \tau \subseteq (\sigma'+v)\cap \tau',
              \]
              so \(\Phi((\sigma,\tau))\leq \Phi((\sigma',\tau'))\). Conversely, if \(F\leq F'\) in \(F(\mathcal{C})\) and \(x\in \relint(F)\subseteq F'\), the minimal cones for \(F\) are contained in those for \(F'\), giving \(\Psi(F)\leq \Psi(F')\). \qedhere
    \end{enumerate}
\end{proof}

Since \(L_v\) is the sublattice of \(P_v\) generated by all atoms, \(L_v\) is isomorphic to the sublattice of \(F(\mathcal{C})\) generated by all vertices. In other words, \(L_v\) is isomorphic to the face lattice of the subcomplex of \(\mathcal{C}\) generated by all vertices. The Möbius function of such lattices is computed in \cite{BayerSturmfels2001}.

\begin{theorem}
    The Möbius function of \(L_v\) is given by
    \[
        \mu_{L_v}\!\bigl(\hat{0},(\sigma,\tau)\bigr)
        = \begin{cases}
            (-1)^{\dim \sigma + \dim \tau - n + 1},
             & \text{if } (\sigma+v)\cap \tau \text{ is bounded}, \\
            0,
             & \text{otherwise}.
        \end{cases}
    \]
\end{theorem}
\begin{proof}
    The value \(\mu_{L_v}\!\bigl(\hat{0},(\sigma,\tau)\bigr)\) depends only on the interval \([\hat{0},(\sigma,\tau)]\) in \(L_v\). Using the isomorphism \(\Phi\), we reduce to the interval between the empty set and the polyhedron \(P=(\sigma+v)\cap \tau\). The interval \([\hat{0},(\sigma,\tau)]\) is isomorphic to the \emph{vertex-facet lattice} \(\hat{\mathcal{V}}(P)\) of~\(P\) defined in \cite{BayerSturmfels2001}.

    By \cite[Theorem~4.4]{BayerSturmfels2001}, the Möbius function of \(\hat{\mathcal{V}}(P)\) is
    \[
        \mu_{\hat{\mathcal{V}}(P)}(\hat{0},P)
        = \begin{cases}
            (-1)^{\dim P+1},
             & \text{if } P \text{ is bounded}, \\
            0,
             & \text{otherwise}.
        \end{cases}
    \]
    By \cref{lem:generic-v-choice},
    \[
        \dim P = \dim\bigl((\sigma+v)\cap \tau\bigr)
        = \dim \sigma + \dim \tau - n,
    \]
    which completes the proof.
\end{proof}

For a nonempty intersection \(P=(\sigma+v)\cap\tau\), its recession cone is
\[
    \operatorname{rec}(P)=\sigma\cap\tau.
\]
Thus \(P\) is bounded if and only if \(\sigma\cap\tau=\{0\}\). Summarizing the results above, we obtain the main product formula.
\begin{theorem}\label{thm:product-rule}
    Let \(\Sigma\) be a complete strongly unimodular fan of dimension~\(n\) in~\(N_{\mathbb{R}}\), and let \(g_1,g_2\) be Grothendieck weights on \(X_\Sigma\). Then
    \[
        (g_1\cdot g_2)(\{0\})
        = \sum_{\substack{\sigma\cap\tau=\{0\}\\(\sigma+v)\cap \tau\neq\emptyset}}
        (-1)^{\dim \sigma + \dim \tau - n}\,
        g_1(\sigma)\,g_2(\tau).
    \]
\end{theorem}

For every pair in this sum, \((\sigma+v)\cap\tau\) is a nonempty bounded polytope, so \((\sigma,\tau)\) belongs to \(L_v\), because a bounded polytope is always the join of its vertices. Therefore the condition \((\sigma,\tau)\in L_v\) can safely be dropped from the summation.

The formula for a general cone \(\gamma\) follows by applying \cref{thm:product-rule} to \(\Star(\gamma)\) in the quotient \(N_\mathbb{R}/\operatorname{span}_\R(\gamma)\). For cones \(\sigma,\tau\supseteq\gamma\), the nonempty condition in the quotient is equivalent to \((\sigma+v)\cap\tau\neq\emptyset\), and the quotient cones intersect only at the origin if and only if \(\sigma\cap\tau=\gamma\); see \cref{rmk:boundedness}. The exponent
\[
    (\dim\sigma-\dim\gamma)+(\dim\tau-\dim\gamma)-(n-\dim\gamma)
\]
has the same parity as \(\dim\sigma+\dim\tau-n+\dim\gamma\).

\begin{theorem}\label{thm:product-rule-arbitrary-cone}
    Let \(\Sigma\) be a complete strongly unimodular fan of dimension~\(n\) in~\(N_{\mathbb{R}}\). Then for any two Grothendieck weights \(g_1,g_2\) on \(X_\Sigma\) and any cone \(\gamma\in \Sigma\),
    \[
        (g_1\cdot g_2)(\gamma)
        = \sum_{\substack{\sigma\cap\tau=\gamma\\(\sigma+v)\cap \tau\neq\emptyset}}
        (-1)^{\dim \sigma + \dim \tau - n + \dim \gamma}\,
        g_1(\sigma)\,g_2(\tau),
    \]
    where the sum ranges over cones \(\sigma,\tau\in\Sigma\).
\end{theorem}

\begin{remark}\label{rmk:boundedness}
    Let \(\sigma,\tau\in\Sigma\) contain \(\gamma\), set \(L=\operatorname{span}_\R(\gamma)\), and let \(\pi:N_\R\to N_\R/L\) be the quotient map. Then
    \[
        \pi\bigl((\sigma+v)\cap\tau\bigr)=(\pi(\sigma)+\pi(v))\cap\pi(\tau).
    \]
    Indeed, if \(x\in\sigma+v\) and \(y\in\tau\) have the same image, write \(x-y=a-b\) with \(a,b\in\gamma\), using \(L=\gamma-\gamma\). Then \(x+b=y+a\) lies in \((\sigma+v)\cap\tau\) and has the same image. Taking \(v=0\) also gives \(\pi(\sigma)\cap\pi(\tau)=\pi(\sigma\cap\tau)\).

    Since \(\gamma\) is a face of \(\sigma\), we have \(\sigma\cap L=\gamma\). Consequently,
    \[
        \pi(\sigma)\cap\pi(\tau)=\{0\}
        \quad\Longleftrightarrow\quad \sigma\cap\tau\subseteq L
        \quad\Longleftrightarrow\quad \sigma\cap\tau=\gamma.
    \]
    When \((\sigma+v)\cap\tau\neq\emptyset\), the recession cone of its image under \(\pi\) is \(\pi(\sigma)\cap\pi(\tau)\). Thus this image is bounded exactly when \(\sigma\cap\tau=\gamma\). This condition depends only on the cones \(\sigma,\tau,\gamma\); the nonempty condition \((\sigma+v)\cap\tau\neq\emptyset\) remains a separate requirement in the sum.
\end{remark}

\begin{example}
    The smallest example is the projective plane \(\mathbb{P}^2\). Its fan has three rays generated by \(\rho_1=(1,0)\), \(\rho_2=(0,1)\), and \(\rho_3=(-1,-1)\). One easily checks that \(N_\sigma + N_\tau = \mathbb{Z}^2\) for any two \(2\)-dimensional cones \(\sigma,\tau\), so the fan of \(\mathbb{P}^2\) is strongly unimodular.

    To write down the product rule, we choose \(v=(1,2)\) as a generic vector. For any two Grothendieck weights \(g_1,g_2\) on \(\mathbb{P}^2\), \cref{thm:product-rule} gives
    \begin{align*}
        (g_1\cdot g_2)(\{0\})
        =\; & g_1(\{0\})\,g_2(\operatorname{cone}(\rho_1,\rho_2))
        + g_1(\rho_3)\,g_2(\rho_2)
        + g_1(\operatorname{cone}(\rho_1,\rho_3))\,g_2(\{0\})        \\
            & - g_1(\rho_3)\,g_2(\operatorname{cone}(\rho_1,\rho_2))
        - g_1(\operatorname{cone}(\rho_1,\rho_3))\,g_2(\rho_2).
    \end{align*}
    In \(K(\mathbb{P}^2)\simeq \mathbb{Z}[x]/(x^3)\), where \(x=[\mathcal{O}_H]\) for a hyperplane \(H\), this becomes
    \begin{align*}
        \chi(g_1\,g_2)
        =\; & \chi(g_1)\,\chi(g_2 x^2)
        + \chi(g_1 x)\,\chi(g_2 x)
        + \chi(g_1 x^2)\,\chi(g_2)         \\
            & - \chi(g_1 x)\,\chi(g_2 x^2)
        - \chi(g_1 x^2)\,\chi(g_2 x).
    \end{align*}
    Setting \(g_1=x^i\), \(g_2=x^j\) for \(i,j=0,1,2\), the left-hand side defines a matrix
    \[
        B = \begin{bmatrix}
            1 & 1 & 1 \\
            1 & 1 & 0 \\
            1 & 0 & 0
        \end{bmatrix},
    \]
    and the right-hand side corresponds to \(BmB\), where
    \[
        m = \begin{bmatrix}
            0 & 0  & 1  \\
            0 & 1  & -1 \\
            1 & -1 & 0
        \end{bmatrix}.
    \]
    Our theorem asserts that \(BmB=B\), which can be directly verified (in fact, \(Bm=I\)).

    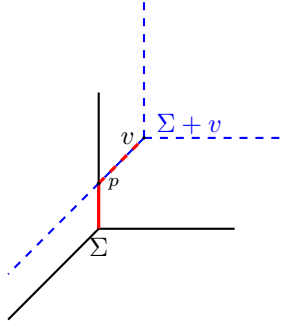
\begin{figure}[ht]
        \centering
        \begin{tikzpicture}[scale=0.6]
            \coordinate (O) at (0,0);
            \coordinate (V) at (1,2);

            \draw[thick] (O) -- (3,0);
            \draw[thick] (O) -- (0,3);
            \draw[line width=1.2pt, red] (0,0) -- (0,1);
            \draw[line width=1.2pt, red,dashed] (0,1) -- (1,2);
            \draw[thick] (O) -- (-2,-2);
            \node[below] at (0,0) {\(\Sigma\)};

            \draw[thick, dashed, blue] (V) -- +(3,0);
            \draw[thick, dashed, blue] (V) -- +(0,3);
            \draw[thick, dashed, blue] (V) -- +(-3,-3);
            \fill (V) circle (1pt) node[left] {\(v\)};
            \node[blue] at (2,2.3) {\(\Sigma+v\)};

            \coordinate (P) at (0,1);
            \fill (P) circle (1.2pt) node[right] {\scriptsize \(p\)};
        \end{tikzpicture}
        \caption{The fan \(\Sigma\) (solid) and its translate \(\Sigma+v\) (dashed).}
    \end{figure}
\end{example}

\section{Grothendieck weights on permutohedral toric varieties}
In this section, we apply the \(K\)-balancing condition and product rule to the permutohedral toric variety and determine the ring structure of Grothendieck weights.

We use the notation introduced in \cref{subsec:permutohedron}. Fix a natural number \(n\), and let \(\Sigma_{[n]}\) be the normal fan of the permutohedron \(\Pi_n\). We use flags of subsets to index cones.

We first determine the \(K\)-balancing condition on \(\Sigma_{[n]}\). Let \(Q=Q_{\Sigma_{[n]}}\) be the lattice polytope defined in \cref{lem:pm1-dual-basis}. By \cref{thm:k-balancing}, a function \(g:\Sigma_{[n]} \to \Z\) is a Grothendieck weight if and only if for each pair \((q,\tau)\) with \(q\in Q\cap M\) and \(\tau\) a \(q\)-neutral cone, one has
\[
    \sum_{\substack{\sigma\supsetneq \tau\\\sigma(1)\setminus\tau(1)\subseteq \mathbf{P}_q}}(-1)^{\dim \sigma}g(\sigma)
    =
    \sum_{\substack{\sigma\supsetneq \tau\\\sigma(1)\setminus\tau(1)\subseteq \mathbf{N}_q}}(-1)^{\dim \sigma}g(\sigma).
\]

\begin{lemma}
    Let \(S=\{e_i-e_j:1\leq i<j\leq n\}\). Then
    \begin{enumerate}
        \item \(S\) is a subset of \(Q\cap M\),
        \item For each \(\tau\in\Sigma_{[n]}\), \(S\cap M(\tau)\) contains an integral basis of \(M(\tau)\).
    \end{enumerate}
\end{lemma}
\begin{proof}
    (1) For any \(1\le i<j\le n\), we have \(e_i-e_j\in M\).
    To show \(e_i-e_j\in Q\), let \(T\subsetneq [n]\) be a nonempty proper subset. The ray \(\rho_T\in\Sigma_{[n]}(1)\) has primitive generator \(e_T\in N\), and
    \[
        \langle e_i-e_j,e_T\rangle
        =
        \begin{cases}
            1,  & i\in T,\ j\notin T, \\
            -1, & i\notin T,\ j\in T, \\
            0,  & \text{otherwise}.
        \end{cases}
    \]
    Hence \(\langle e_i-e_j,e_T\rangle\in\{-1,0,1\}\) for all rays, so \(e_i-e_j\in Q\). Therefore \(S\subseteq Q\cap M\).

    (2) Fix \(\tau=\sigma_{\mathcal F}\), where
    \[
        \mathcal F:\emptyset\subsetneq F_1\subsetneq\cdots\subsetneq F_r\subsetneq [n].
    \]
    Let
    \[
        B_1=F_1,\quad B_2=F_2\setminus F_1,\quad \dots,\quad B_r=F_r\setminus F_{r-1},\quad B_{r+1}=[n]\setminus F_r.
    \]
    Then
    \[
        M(\tau)=\left\{a=(a_1,\dots,a_n)\in M:\sum_{u\in B_t}a_u=0\ \text{for all }t=1,\dots,r+1\right\}.
    \]
    Hence
    \[
        M(\tau)=\bigoplus_{t=1}^{r+1}L_t,\qquad
        L_t:=\{a\in\Z^n:\operatorname{supp}(a)\subseteq B_t,\ \sum_{u\in B_t}a_u=0\}.
    \]
    For each block \(B_t\), choose \(b_t=\min(B_t)\) and define
    \[
        \mathcal B_t:=\{\,e_{b_t}-e_u\mid u\in B_t\setminus\{b_t\}\,\}.
    \]
    Every vector in \(\mathcal B_t\) lies in \(S\), and also in \(M(\tau)\).
    Moreover, \(\mathcal B_t\) is a \(\Z\)-basis of \(L_t\). Therefore
    \[
        \mathcal B:=\bigcup_{t=1}^{r+1}\mathcal B_t
    \]
    is a \(\Z\)-basis of \(M(\tau)\) and satisfies \(\mathcal B\subseteq S\cap M(\tau)\). We conclude that \(S\cap M(\tau)\) contains an integral basis of \(M(\tau)\).
\end{proof}

By \cref{rmk:smaller-Q}, the \(K\)-balancing condition on \(\Sigma_{[n]}\) can be indexed by \(S\). To apply \cref{thm:k-balancing} to \(\Sigma_{[n]}\), we first translate the theorem using the language of flags of subsets. We summarize the result in the following lemma.

\begin{lemma}
    Suppose \(\tau=\sigma_{\mathcal{G}}\), \(\sigma=\sigma_{\mathcal{F}}\), and \(q=e_i-e_j\). Then
    \begin{enumerate}
        \item \(\tau\) is \(q\)-neutral if and only if for each \(G_i\in \mathcal{G}\), either \(\{i,j\}\cap G_i=\emptyset\), or \(\{i,j\}\subset G_i\). We call such a flag of subsets \(\{i,j\}\)-neutral.
        \item \(\tau\subset\sigma\) if and only if \(\mathcal{F}\) refines \(\mathcal{G}\), and we denote this by \(\mathcal{F}\succeq \mathcal{G}\).
        \item \(\sigma(1)\setminus\tau(1)\subseteq \mathbf{P}_q\) if and only if \(i\in H, j\not\in H\) for all \(H\in \mathcal{F}\setminus \mathcal{G}\).
        \item \(\sigma(1)\setminus\tau(1)\subseteq \mathbf{N}_q\) if and only if \(i\not\in H, j\in H\) for all \(H\in \mathcal{F}\setminus \mathcal{G}\).
        \item \(\dim \sigma=\ell(\mathcal{F})\coloneqq |\mathcal{F}|\), the length of the flag.
    \end{enumerate}
\end{lemma}

\permutohedralBalancing*

The product rule for Grothendieck weights follows from \cref{thm:product-rule-arbitrary-cone}. For the reader's convenience, we restate \cref{cor:product-rule-general} here.

\permutohedralProduct*

\begin{remark}\label{rmk:generic-permutohedron}
    Let \((v_1,\dots,v_n)\) represent \(v\in N_\R\). Then \(v\) is generic for \(\Sigma_{[n]}\) in the sense of \cref{lem:generic-v-choice} if
    \[
        \sum_{i\in S}v_i\neq\sum_{i\in T}v_i
        \qquad\text{for all distinct }S,T\subseteq[n]\text{ with }|S|=|T|.
    \]
    This condition is unchanged by adding a common constant to the coordinates, so it is well defined on \(N_\R\). In particular, \(v_i=2^i\) gives a generic vector.
\end{remark}
\begin{proof}
    By the proof of \cref{lem:generic-v-choice}, it suffices to show that \(v\) avoids every proper sum \(L_\sigma+L_\tau\). Fix such \(\sigma,\tau\). Form a matrix whose rows are the indicator representatives of their primitive ray generators, together with \(\mathbf 1=(1,\ldots,1)\). Apply the row operations in \cref{eg:braid_is_strongly_unimodular}, and denote the resulting totally unimodular matrix by \(A\). Its row space is the inverse image of \(L_\sigma+L_\tau\) in \(\R^n\), so
    \[
        v\in L_\sigma+L_\tau
        \iff\operatorname{rank}\begin{pmatrix}A\\v\end{pmatrix}=\operatorname{rank}A.
    \]

    Let \(r=\operatorname{rank}A<n\), and choose \(r\) independent rows to form a matrix \(B\). Choose \(r+1\) columns \(I\) containing a nonzero \(r\times r\) minor of \(B\). Expanding along the last row, the total unimodularity of \(A\) gives
    \[
        \det\begin{pmatrix}B_I\\v_I\end{pmatrix}
        =\sum_{i\in I}c_i v_i,\qquad c_i\in\{0,\pm1\},
    \]
    where the cofactors \(c_i\) are not all zero. The vector \(\mathbf 1\) belongs to the row space of \(A\), hence also of \(B\). Therefore
    \[
        \sum_{i\in I}c_i=\det\begin{pmatrix}B_I\\\mathbf 1_I\end{pmatrix}=0.
    \]
    Thus \(S=\{i:c_i=1\}\) and \(T=\{i:c_i=-1\}\) are nonempty and have equal cardinality, and the hypothesis gives
    \[
        \det\begin{pmatrix}B_I\\v_I\end{pmatrix}
        =\sum_{i\in S}v_i-\sum_{i\in T}v_i\neq0
    \]
    for this \(I\). Therefore, \(v\notin L_\sigma+L_\tau\).
\end{proof}

For the rest of the section, we compute Grothendieck weights of nef line bundles on \(X\).

Recall that nef line bundles on \(X=X_{[n]}\) are in one-to-one correspondence with generalized permutohedra. Let \(P\) be any generalized permutohedron, and let \(\mathcal{L}_P\) be the corresponding nef line bundle. Let \(g_P\) denote the Grothendieck weight defined by \([\mathcal{L}_P]\in K(X)\), i.e.,
\[
    g_P(\mathcal{F})=\chi(\mathcal{L}_P \mathcal{O}_{V(\sigma)})=\chi(\mathcal{L}_P|_{V(\sigma)}).
\]

To express \(g_P\) purely in terms of \(P\), we need to introduce the following terminology.

\begin{definition}
    Let \(P\subset M_\R\) be a generalized permutohedron, and let \(h_P(u):=\min_{m\in P}\langle m,u\rangle\) denote its divisor support function. For a flag of subsets \(\mathcal{F}:\emptyset\subsetneq F_1\subsetneq\cdots\subsetneq F_k\subsetneq[n]\), the \(\mathcal{F}\)-\emph{face} of \(P\) is
    \[
        \face_{\mathcal{F}}(P):=\{m\in P:\langle m,u\rangle=h_P(u)\}
    \]
    for any \(u\) in the relative interior of \(\sigma_{\mathcal{F}}=\operatorname{cone}(e_{F_1},\dots,e_{F_k})\). This is independent of the choice of \(u\) since the normal fan of \(P\) coarsens \(\Sigma_{[n]}\). For the empty flag \(\mathcal{F}=\emptyset\), set \(\face_\emptyset(P):=P\).
\end{definition}

It is well known that on a complete toric variety, for any nef line bundle with corresponding polytope \(P\),
\[
    \chi(\mathcal{L}_P)=|P\cap M|,
\]
which follows from \cite[Proposition~4.3.3 and Theorem~9.2.1]{CoxLittleSchenck2011}. Applying the same argument to \(\mathcal{L}_P|_{V(\sigma)}\), we obtain
\begin{lemma}\label{lem:gw-generalized-permutohedra}
    For a generalized permutohedron \(P\) in \(M_\R\), the Grothendieck weight \(g_P\) is given by
    \[
        g_P(\mathcal{F})=|\face_{\mathcal{F}}(P)\cap M|.
    \]
\end{lemma}
\begin{proof}
    Let \(\sigma=\sigma_{\mathcal{F}}\) and \(Y:=V(\sigma)\subset X_{[n]}\). By definition,
    \[
        g_P(\mathcal{F})=\chi(\mathcal{L}_P|_Y).
    \]
    We need to show that the polytope corresponding to \(\mathcal{L}_P|_Y\) is \(\face_\mathcal{F}(P)\) up to a translation.

    Let \(D_P\) be the torus-invariant Cartier divisor corresponding to \(P\), so \(\mathcal{L}_P=\mathcal{O}_{X_{[n]}}(D_P)\). Let \(h_P|_\sigma\) be the restriction of \(h_P\) to \(\sigma\). Since \(h_P|_\sigma\) is linear, we can choose \(h_{P,\sigma}^\ast\in M\) such that
    \[
        h_P(u)=\langle h_{P,\sigma}^\ast,u\rangle \quad \text{for all }u\in \sigma.
    \]
    (Here we use \(M=\Hom(N,\Z)\) to identify integral linear functions with lattice elements.) The character lattice of the dense torus of \(Y=V(\sigma)\) is
    \[
        M(\sigma)=M\cap \sigma^\perp,
    \]
    and the fan of \(Y\) is the star fan of \(\Sigma_{[n]}\) in \(N(\sigma)_\R:=N_\R/(N_\sigma)_\R\). By \cite[Proposition~6.2.7]{CoxLittleSchenck2011}, the support function of \(D_P|_Y\) is
    \[
        h_P^\sigma(\bar u)=h_P(u)-\langle h_{P,\sigma}^\ast,u\rangle,\qquad \bar u\in N(\sigma)_\R,
    \]
    where \(u\) is a lift of \(\bar u\) chosen in a cone \(\tau\supseteq\sigma\) of \(\Sigma_{[n]}\). We claim that this is exactly the support function of the translated face
    \[
        \face_{\mathcal{F}}(P)-h_{P,\sigma}^\ast \subset M(\sigma)_\R.
    \]

    Let \(F=\face_{\mathcal{F}}(P)\). The support function of the above face is
    \[
        \begin{aligned}
            h_{F-h_{P,\sigma}^\ast}(\bar u) & =\min_{m\in F-h_{P,\sigma}^\ast}\langle m,\bar u \rangle             \\
                                            & =\min_{m\in F}\langle m,u\rangle-\langle h_{P,\sigma}^\ast,u \rangle \\
        \end{aligned}
    \]
    Since \(u\in\tau\), the minimum \(\min_{m\in P}\langle m,u \rangle \) is obtained on the face \(\face_\tau(P)\subset \face_\sigma(P)=F\). Therefore,
    \[
        \min_{m\in F}\langle m,u\rangle-\langle h_{P,\sigma}^\ast,u \rangle=\min_{m\in P}\langle m,u\rangle-\langle h_{P,\sigma}^\ast,u \rangle=h^\sigma_P(\bar u).
    \]

    Hence, by \cite[Proposition~4.3.3]{CoxLittleSchenck2011},
    \[
        h^0(Y,\mathcal{L}_P|_Y)=|(\face_{\mathcal{F}}(P)-h_{P,\sigma}^\ast)\cap M(\sigma)|.
    \]
    Since \(\face_{\mathcal{F}}(P)\subset h_{P,\sigma}^\ast+M(\sigma)_\R\),
    \[
        h^0(Y,\mathcal{L}_P|_Y)=|(\face_{\mathcal{F}}(P)-h_{P,\sigma}^\ast)\cap M|=|\face_{\mathcal{F}}(P)\cap M|.
    \]
    Since \(\mathcal{L}_P\) is nef on \(X_{[n]}\), its restriction to \(Y\) is nef. By Demazure vanishing \cite[Theorem~9.2.1]{CoxLittleSchenck2011} on the complete toric variety \(Y\),
    \[
        H^i(Y,\mathcal{L}_P|_Y)=0\quad (i>0).
    \]
    Therefore
    \[
        g_P(\mathcal{F})=\chi(\mathcal{L}_P|_Y)=h^0(Y,\mathcal{L}_P|_Y)=|\face_{\mathcal{F}}(P)\cap M|.
    \]
\end{proof}

\begin{corollary}\label{cor:simplicial-generator-gw}
    Let \(\emptyset\neq I\subseteq [n]\), \(i_0=\min(I)\), and set
    \[
        \Delta_I:=\operatorname{conv}\!\bigl(\{e_i-e_{i_0}:i\in I\}\bigr)\subset M_\R.
    \]
    Then \(\Delta_I\) is a generalized permutohedron. For a flag \(\mathcal{F}:\emptyset=F_0 \subsetneq F_1\subsetneq\cdots\subsetneq F_r\subsetneq F_{r+1} = [n]\),
    let \(t_I(\mathcal{F})\) be the smallest index \(t\) such that \(I\subset F_t\). Then
    \[
        g_{\Delta_I}(\mathcal{F})=|I\cap (F_{t_I(\mathcal{F})}\backslash F_{t_I(\mathcal{F})-1})|.
    \]
\end{corollary}
\begin{proof}
    Write \(u=\sum_{t=1}^r\lambda_t e_{F_t}\in\sigma^\circ_{\mathcal{F}}\) with \(\lambda_t>0\). For \(i\in F_t\setminus F_{t-1}\),
    \[
        u_i:=\sum_{s\ge t}\lambda_s
    \]
    gives a representative of \(u\) in \(\R^n\) modulo the line \(\R(1,\dots,1)\). The values \(u_i\) are constant within each block and strictly decreasing as \(t\) increases. Therefore
    \[
        h_{\Delta_I}(u)=\min_{i\in I}(u_i-u_{i_0})
    \]
    is linear on \(\sigma_{\mathcal F}\), so the normal fan of \(\Delta_I\) coarsens \(\Sigma_{[n]}\). Thus \(\Delta_I\) is a generalized permutohedron. Moreover, for \(u\in\sigma^\circ_{\mathcal F}\), the minimum is achieved exactly by \(i\in I\cap (F_{t_I(\mathcal{F})}\backslash F_{t_I(\mathcal{F})-1})\), giving
    \[
        \face_{\mathcal{F}}(\Delta_I)=\operatorname{conv}\!\bigl(\{e_i-e_{i_0}:i\in I\cap (F_{t_I(\mathcal{F})}\backslash F_{t_I(\mathcal{F})-1})\bigr).
    \]

    Since \(\face_{\mathcal{F}}(\Delta_I)\) is a unimodular simplex, its only lattice points are its vertices, hence \[|\face_{\mathcal{F}}(\Delta_I)\cap M|=|I\cap (F_{t_I(\mathcal{F})}\backslash F_{t_I(\mathcal{F})-1})|.\]
\end{proof}

\section{The Grothendieck ring of matroids and Grothendieck weights}
In this section, we introduce the $K$-ring of matroids and Grothendieck weights on matroidal fans. Throughout this section, $\mathsf{M}$ denotes a \emph{loopless} matroid on $E=[n]$ of rank $r$.

The unaugmented $K$-ring of loopless matroids was introduced in \cite{LARSON2024109554}.
\begin{definition}[{\cite[Theorem~5.2]{LARSON2024109554}}]
    For a loopless matroid $\mathsf{M}$ on $E$, its $K$-ring $K(\mathsf{M})$ is generated by the symbols $x_F$ for each nonempty proper flat $F$ of $\mathsf{M}$, subject to the following relations:
    \begin{itemize}
        \item $x_Fx_G=0$ if $F$ and $G$ are incomparable;
        \item For any two elements $i,j\in E$, we have
        \[
            \prod_{i\not\in F}(1-x_F)=\prod_{j\not\in F}(1-x_F).
        \]
    \end{itemize}
\end{definition}

For each matroid $\mathsf{M}$, one associates a matroidal fan $\Sigma_{\mathsf{M}}$ (see, e.g., \cite[Section~3]{ArdilaKlivans2006}), whose rays are
\[
    \Sigma_{\mathsf{M}}(1)=\{\rho_F\mid F \text{ is a nonempty proper flat}\},
\]
and whose $k$-dimensional cones are labeled by flags of nonempty proper flats
\[\mathcal{F}:\emptyset=F_0\subsetneq F_1\subsetneq\cdots\subsetneq F_k\subsetneq F_{k+1}=E.\]
The number $k$ is called the length of the flag, denoted by $\ell(\mathcal{F})$. We also identify $\Sigma_{\mathsf{M}}$ with the set of flags of nonempty proper flats of $\mathsf{M}$.

By \cite[Theorem~5.2]{LARSON2024109554}, the ring \(K(\mathsf{M})\) is canonically isomorphic to \(K(X_{\Sigma_{\mathsf{M}}})\). Under this identification, the linear relation in the definition may be written as
\[
    \prod_{i\not\in F,j\in F}(1-x_F)=\prod_{j\not\in F,i\in F}(1-x_F),
\]
which is the toric linear relation corresponding to \(e_i-e_j\in M\). We also use the Euler characteristic map \(\chi:K(\mathsf{M})\to\Z\) and the perfect pairing on \(K(\mathsf{M})\) constructed in \cite[Section~1.5]{LARSON2024109554}:
\[
    \langle a,b \rangle = \chi(ab)
\]
for \(a,b\in K(\mathsf{M})\).

For an element $a\in K(\mathsf{M})$, we define the associated Grothendieck weight \(g_a\) as in the case of toric varieties. For a flag of flats $\mathcal{F}$, set
\[
    g_a(\mathcal{F})=\chi(ax_{\mathcal{F}}).
\]
For a flag of flats $\mathcal{F}=\emptyset\subsetneq F_1\subsetneq\cdots\subsetneq F_k \subsetneq E$, we write $x_{\mathcal{F}}=x_{F_1}x_{F_2}\cdots x_{F_k}$, and $x_\emptyset=1$.

\begin{definition}\label{def:GW-on-matroids}
    A map $g:\Sigma_{\mathsf{M}}\to \Z$ is a Grothendieck weight on $\Sigma_{\mathsf{M}}$ if for every linear relation among the $x_{\mathcal{F}}$
    \[
        \sum_{\mathcal{F}} c_{\mathcal{F}} x_{\mathcal{F}} = 0,
    \]
    one has
    \[
        \sum_{\mathcal{F}} c_{\mathcal{F}} g(\mathcal{F})=0.
    \]
\end{definition}

The abelian group of Grothendieck weights on $\Sigma_{\mathsf{M}}$ is denoted by $\GW(\mathsf{M})$.

\begin{proposition}
    For a loopless matroid $\mathsf{M}$ on $E=[n]$, the map defined by
    \[
        \Phi:K(\mathsf{M})\to \GW(\mathsf{M}),\qquad \Phi: a\mapsto g_a
    \]
    is an isomorphism of abelian groups.
\end{proposition}
\begin{proof}
    The proof is identical to that of \cref{prop:K-bijection-GW}.
\end{proof}

We say that a flag of flats \(\mathcal{G}\) is \(\{i,j\}\)-neutral if every flat in \(\mathcal{G}\) either contains both \(i\) and \(j\) or contains neither.

\matroidBalancing*
\begin{proof}
    Let \(J\subset K(X_{[n]})\) be the ideal generated by the classes \(x_S\), where \(S\subsetneq E\) is a nonempty proper subset that is not a flat of \(\mathsf{M}\). Comparing the presentation of \(K(X_{[n]})\) in \cref{thm:K_ring_toric} with the presentation of \(K(\mathsf{M})\) above, we obtain a canonical isomorphism
    \[
        K(X_{[n]})/J\xrightarrow{\sim} K(\mathsf{M}),
    \]
    and we write \(\pi:K(X_{[n]})\twoheadrightarrow K(\mathsf{M})\) for the quotient map.

    We also note that every element of \(J\) is an integral linear combination of the classes \(x_{\mathcal{H}}\), where \(\mathcal{H}\) contains a non-flat subset. Indeed, for such a subset \(S\), let \(i_S:V(\rho_S)\hookrightarrow X_{[n]}\) be the corresponding invariant divisor. Since \(x_S=i_{S*}\mathcal{O}_{V(\rho_S)}\), the projection formula gives
    \[
        x_S\cdot a=i_{S*}i_S^*(a)
    \]
    for every \(a\in K(X_{[n]})\). As \(K(V(\rho_S))\) is additively generated by its invariant subvarieties, the ideal \((x_S)\) is contained in \(\mathrm{span}_\Z\{x_{\mathcal{H}}:S\in\mathcal{H}\}\). The reverse inclusion holds since \(x_{\mathcal{H}}=x_S\cdot x_{\mathcal{H}\setminus\{S\}}\in(x_S)\) for any \(S\in\mathcal{H}\). Summing over all non-flat \(S\) gives the claim.

    Let \(\widetilde{g}:\Sigma_{[n]}\to\Z\) be the zero-extension of \(g\). We claim that \(g\in \GW(\mathsf{M})\) if and only if \(\widetilde{g}\in \GW(\Sigma_{[n]})\). Suppose first that \(\widetilde{g}\in \GW(\Sigma_{[n]})\), and let
    \[
        \sum_{\mathcal{F}} c_{\mathcal{F}}x_{\mathcal{F}}=0
    \]
    be a linear relation in \(K(\mathsf{M})\), where the sum is over flat flags. Viewing the same sum in \(K(X_{[n]})\), it lies in \(\ker(\pi)=J\). Hence, by the previous paragraph, there exist integers \(d_{\mathcal{H}}\) such that
    \[
        \sum_{\mathcal{F}} c_{\mathcal{F}}x_{\mathcal{F}}=\sum_{\mathcal{H}} d_{\mathcal{H}}x_{\mathcal{H}}
    \]
    in \(K(X_{[n]})\), where each \(\mathcal{H}\) contains a non-flat subset. Applying \(\widetilde{g}\) to the resulting relation
    \[
        \sum_{\mathcal{F}} c_{\mathcal{F}}x_{\mathcal{F}}-\sum_{\mathcal{H}} d_{\mathcal{H}}x_{\mathcal{H}}=0
    \]
    gives \(\sum_{\mathcal{F}} c_{\mathcal{F}}g(\mathcal{F})=0\), since \(\widetilde{g}(\mathcal{H})=0\) for every non-flat flag \(\mathcal{H}\). Thus \(g\in\GW(\mathsf{M})\). Conversely, if \(g\in\GW(\mathsf{M})\) and \(\sum_{\mathcal{K}} c_{\mathcal{K}}x_{\mathcal{K}}=0\) in \(K(X_{[n]})\), then applying \(\pi\) gives a relation \(\sum_{\mathcal{F}} c_{\mathcal{F}}x_{\mathcal{F}}=0\) in \(K(\mathsf{M})\), where the sum is over flat flags. Therefore
    \[
        \sum_{\mathcal{K}} c_{\mathcal{K}}\widetilde{g}(\mathcal{K})=\sum_{\mathcal{F}} c_{\mathcal{F}}g(\mathcal{F})=0,
    \]
    so \(\widetilde{g}\in\GW(\Sigma_{[n]})\).

    By \cref{thm:k-balancing-permutohedron}, \(\widetilde{g}\in \GW(\Sigma_{[n]})\) if and only if it satisfies the \(\{i,j\}\)-balancing condition for every \(\{i,j\}\)-neutral flag of subsets \(\mathcal{G}\). If \(\mathcal{G}\) contains a non-flat subset, then so does every refinement \(\mathcal{F}\sref\mathcal{G}\), hence \(\widetilde{g}(\mathcal{F})=0\) and both sides are zero. If \(\mathcal{G}\) is a flag of flats, then \(\widetilde{g}(\mathcal{F})=g(\mathcal{F})\) for flat refinements and \(0\) otherwise, so this is exactly the balancing condition in the statement.
\end{proof}

The element $1\in K(\mathsf{M})$ defines a Grothendieck weight \(\Delta_{\mathsf{M}}:\Sigma_{\mathsf{M}}\to\Z\). It maps every flag $\mathcal{F}$ to $1$ because
\[
    \chi(x_{\mathcal{F}})=1
\]
for every flag of flats \(\mathcal{F}\), by the defining property of the Euler characteristic map in \cite[Section~1.5]{LARSON2024109554}. This gives the following combinatorial identity.

\begin{corollary}
    For a loopless matroid $\mathsf{M}$,
    \[
        \sum_{\mathcal{F}\in\mathbf{S}_{ij}(\emptyset)}(-1)^{\ell(\mathcal{F})} = \sum_{\mathcal{F}\in\mathbf{S}_{ji}(\emptyset)}(-1)^{\ell(\mathcal{F})}.
    \]
\end{corollary}

The groups of Grothendieck weights $\GW(\mathsf{M})$ and $\GW(\Sigma_{[n]})$ are related by zero-extension.
\begin{proposition}
    The zero-extension map $\iota_*:\GW(\mathsf{M})\to \GW(\Sigma_{[n]})$ defined by $\iota_*(g)=\tilde{g}$, where
    \[
        \tilde{g}(\mathcal{F})=\begin{cases}
            g(\mathcal{F}),& \text{ if }\mathcal{F}\text{ is a flag of nonempty proper flats},\\
            0,& \text{ otherwise}
        \end{cases}
    \]
    is well-defined. Here $\mathcal{F}$ is a flag of nonempty proper subsets.
\end{proposition}
\begin{proof}
    This is immediate from \cref{prop:K-balancing-for-matroid}.
\end{proof}

We also need the following geometric interpretation of $K(\mathsf{M})$ in the realizable case.

Assume that \(\mathsf{M}\) is realizable over \(\C\), and let \(L\subset \C^E\) be a linear subspace that realizes \(\mathsf{M}\). Set
\[
    U_L=\PP(L\cap (\C^*)^E),
\]
and let \(W_L\) be the closure of \(U_L\) in \(X_{[n]}\). This is the wonderful compactification associated with the realization \(L\).

\begin{proposition}[{\cite[Proposition~1.6]{LARSON2024109554}}]\label{prop:K-matroid-wonderful}
    Let \(\mathsf{M}\) be a loopless realizable matroid, and let \(L\subset \C^E\) be a realization of \(\mathsf{M}\). Then \(W_L\) is contained in \(X_{\Sigma_{\mathsf{M}}}\), and the natural inclusion
    \[
        W_L\hookrightarrow X_{\Sigma_{\mathsf{M}}}
    \]
    induces an isomorphism
    \[
        K(X_{\Sigma_{\mathsf{M}}})\xrightarrow{\sim} K(W_L).
    \]
    Consequently, combining this with \cite[Theorem~5.2]{LARSON2024109554}, there is a canonical isomorphism
    \[
        K(\mathsf{M})\xrightarrow{\sim} K(W_L).
    \]
\end{proposition}

In particular, in the realizable case we may freely regard an element of \(K(\mathsf{M})\) as a class on the wonderful model \(W_L\), and conversely.

\section{Motivic Chern classes of matroids}
Throughout this section, we assume $\mathsf{M}$ is a loopless matroid on $E=[n]$ of rank $r$. We first assume $\mathsf{M}$ is realizable over $\C$. Let $L\subset \C^E$ be a linear subspace that realizes $\mathsf{M}$. Set
\[
    U_L=\PP(L\cap (\C^*)^E),
\]
and let $W_L$ be the closure of \(U_L\) in $X_{[n]}$. Our goal is to give a combinatorial description of the motivic Chern class of $\mathsf{M}$, which we define below after reviewing the necessary background.

\subsection{Motivic Chern classes}
Following \cite{BrasseletSchurmannYokura2010}, for a complex algebraic variety \(X\) let \(K_0(\mathrm{var}/X)\) be the relative Grothendieck group of algebraic varieties over \(X\). It is generated by isomorphism classes \([f:Y\to X]\), subject to the additivity relation
\[
    [Y\to X]=[Z\to X]+[Y\setminus Z\to X]
\]
for every closed subvariety \(Z\subseteq Y\). A proper morphism \(p:X\to X'\) induces a pushforward
\[
    p_*:K_0(\mathrm{var}/X)\to K_0(\mathrm{var}/X'),
    \qquad
    [f:Y\to X]\mapsto [p\circ f:Y\to X'].
\]

Brasselet, Sch\"urmann, and Yokura constructed a natural transformation
\[
    mC_y:K_0(\mathrm{var}/X)\longrightarrow G_0(X)[y]
\]
commuting with proper pushforward and characterized by the normalization
\[
    mC_y([\mathrm{id}_X])=\lambda_y[\Omega_X]
\]
for smooth \(X\). Here \(G_0(X)\) denotes the Grothendieck group of coherent sheaves on \(X\). By the convention in \cref{subsec:toric-k-ring}, the targets used in this paper have \(G_0(X)\simeq K(X)\), and we write
\[
    \MC_y(Y\to X)\coloneq mC_y([Y\to X])\in K(X)[y].
\]
With this convention, additivity in \(K_0(\mathrm{var}/X)\) becomes
\[
    \MC_y(Y\to X)=\MC_y(Z\to X)+\MC_y(Y\setminus Z\to X),
\]
and proper functoriality becomes
\[
    p_*\MC_y(Y\to X)=\MC_y(Y\to X').
\]
In particular, if \(X\) is smooth, then
\[
    \MC_y(X\xrightarrow{\mathrm{id}} X)=\lambda_y[\Omega_X].
\]

To compare with classical characteristic classes, one usually passes from motivic Chern classes to the associated Hirzebruch class transformation
\[
    T_{y*}:=td_{(1+y)}\circ \MC_y:K_0(\mathrm{var}/X)\longrightarrow H_*(X)\otimes \Q[y],
\]
where \(td_{(1+y)}\) denotes the normalized Baum--Fulton--MacPherson Todd transformation; see \cite{BrasseletSchurmannYokura2010}. The specializations at \(y=-1,0,1\) recover the classical singular characteristic classes:
\[
    T_{-1,*}=c_{\mathrm{SM}}\otimes \Q,\qquad
    T_{0,*}=td_*,\qquad
    T_{1,*}=L_*.
\]
In particular, the relation between motivic Chern classes and the Chern--Schwartz--MacPherson class is mediated by \(T_{y*}\): one first applies \(td_{(1+y)}\) to \(\MC_y\), and then specializes at \(y=-1\).

\begin{definition}\label{def:motivic-Chern-realizable}
    The \emph{motivic Chern class} of a matroid $\mathsf{M}$ realizable over $\C$ is defined to be
    \[
        \MC_y(\mathsf{M})\coloneq \MC_y(U_L\to W_L)\in K(W_L)[y],
    \]
    where $L$ is any subspace in $\C^E$ realizing $\mathsf{M}$.
\end{definition}

The computations below determine the Grothendieck weight associated with \(\mathbb{D}_{W_L}(\MC_y(\mathsf{M}))\). The independence of the choice of realization \(L\) is explained at \cref{rmk:independence-of-L}.

\subsection{Motivic Chern class of the complement of an SNC divisor}

The boundary $W_L\setminus U_L$ is an SNC (simple normal crossing) divisor. We need the following result on the motivic Chern class of the complement of an SNC divisor.

\begin{theorem}[{\cite[Theorem~5.1 and Remark~5.2]{Weber2016}}]
    \label{thm:MCy-SNC-U}
    Let $X$ be smooth, $D=\bigcup_i D_i$ an SNC divisor, and $U=X\setminus D$. Then
    \begin{equation}
        \label{eq:MCy-log-formula}
        \MC_y(U\hookrightarrow X)
        \;=\;
        \lambda_y [\Omega_X(\log D)] \cdot \mathcal{O}_X(-D)
        \quad\in K(X)[y],
    \end{equation}
    where $\Omega_X(\log D)$ denotes the sheaf of logarithmic $1$-forms with poles along $D$.
\end{theorem}

\begin{proposition}
    \label{prop:MCy-SNC-pairing}
    Let $X$ be a smooth proper variety, $D=\bigcup_{i=1}^m D_i$ an SNC divisor, and $U=X\setminus D$.
    For each $I\subseteq\{1,\ldots,m\}$, set $D_I=\bigcap_{i\in I}D_i$, let $i_I:D_I\hookrightarrow X$ be the closed embedding, let $D_I^\circ=D_I\setminus\bigcup_{j\notin I}D_j$, and let $N_{D_I/X}$ denote the normal bundle of $D_I$ in $X$. Then
    \begin{equation}
        \label{eq:MCy-pairing}
        \chi\bigl(\MC_y(U\to X)\cdot i_{I*}\det N_{D_I/X}\bigr)=(1+y)^{|I|}\chi_y(D_I^\circ),
    \end{equation}
    where $\chi_y$ denotes the Hirzebruch $\chi_y$-genus.
\end{proposition}
\begin{proof}
    By the projection formula,
    \[
        \chi\bigl(\MC_y(U\to X)\cdot i_{I*}\det N_{D_I/X}\bigr)
        =\chi\bigl(i_I^*\MC_y(U\to X)\cdot\det N_{D_I/X}\bigr).
    \]
    We claim that
    \begin{equation}
        \label{eq:claim-MCy-pullback}
        i_I^*\MC_y(U\to X)\cdot\det N_{D_I/X}=(1+y)^{|I|}\MC_y(D_I^\circ\to D_I).
    \end{equation}
    Granting \cref{eq:claim-MCy-pullback}, the proposition follows from the standard identity
    \[
        \chi\bigl(\MC_y(V\to Y)\bigr)=\chi_y(V)
    \]
    applied to $D_I^\circ\hookrightarrow D_I$.

    It remains to prove \cref{eq:claim-MCy-pullback}. Set $D|_{D_I}=\bigcup_{j\notin I}D_j|_{D_I}$, an SNC divisor on $D_I$ with $D_I^\circ=D_I\setminus D|_{D_I}$. Applying \cref{thm:MCy-SNC-U} gives
    \begin{align}
        \MC_y(U\to X)           & = \lambda_y[\Omega_X(\log D)]\cdot\mathcal{O}_X(-D), \label{eq:MCy-U}                        \\
        \MC_y(D_I^\circ\to D_I) & = \lambda_y[\Omega_{D_I}(\log D|_{D_I})]\cdot\mathcal{O}_{D_I}(-D|_{D_I}). \label{eq:MCy-DI}
    \end{align}
    Since $D$ is SNC, $N_{D_I/X}=\bigoplus_{i\in I}\mathcal{O}_X(D_i)|_{D_I}$, so $\det N_{D_I/X}=i_I^*\mathcal{O}_X(\sum_{i\in I}D_i)$. The line bundle factor in \cref{eq:MCy-U} therefore satisfies
    \[
        i_I^*\mathcal{O}_X(-D)\cdot\det N_{D_I/X}
        = i_I^*\mathcal{O}_X\bigl(-\textstyle\sum_{j\notin I}D_j\bigr)
        = \mathcal{O}_{D_I}(-D|_{D_I}).
    \]
    For the logarithmic cotangent bundle, the residue maps $\mathrm{Res}_{D_i}:\Omega_X(\log D)\to \mathcal{O}_{D_i}$ for $i\in I$, restricted to $D_I$, assemble into a short exact sequence
    \[
        0\to\Omega_{D_I}(\log D|_{D_I})\to i_I^*\Omega_X(\log D)
        \xrightarrow{\,\bigoplus_{i\in I}\mathrm{Res}_{D_i}\,}
        \mathcal{O}_{D_I}^{\oplus|I|}\to 0.
    \]
    In $K(D_I)$, this gives $[i_I^*\Omega_X(\log D)]=[\Omega_{D_I}(\log D|_{D_I})]+|I|[\mathcal{O}_{D_I}]$. Since $\lambda_y$ is multiplicative in $K$-theory and $\lambda_y[\mathcal{O}_{D_I}]=1+y$, this gives
    \begin{equation}
        \label{eq:lambda-split}
        \lambda_y[i_I^*\Omega_X(\log D)]=(1+y)^{|I|}\,\lambda_y[\Omega_{D_I}(\log D|_{D_I})].
    \end{equation}
    Combining \cref{eq:MCy-U,eq:lambda-split} and the line bundle computation,
    \[
        i_I^*\MC_y(U\to X)\cdot\det N_{D_I/X}
        =(1+y)^{|I|}\,\lambda_y[\Omega_{D_I}(\log D|_{D_I})]\cdot\mathcal{O}_{D_I}(-D|_{D_I}).
    \]
    By \cref{eq:MCy-DI}, the right-hand side equals $(1+y)^{|I|}\MC_y(D_I^\circ\to D_I)$, proving \cref{eq:claim-MCy-pullback}.
\end{proof}

By the adjunction formula,
\[
    \omega_{D_I}\cong i_I^*\omega_X\otimes \det N_{D_I/X}.
\]
This allows us to reformulate \cref{prop:MCy-SNC-pairing} in terms of normalized duality.

\begin{corollary}\label{cor:GW-for-MCy}
    Let \(n=\dim X\), and let \(\mathbb{D}_X:K(X)[y]\to K(X)[y]\) be the normalized duality operator defined coefficientwise by
    \[
        \mathbb{D}_X([\mathcal{F}])=[R\sHom(\mathcal{F},\omega_X[n])].
    \]
    Then, for every \(I\subseteq\{1,\ldots,m\}\),
    \[
        \chi\bigl(\mathbb{D}_X(\MC_y(U\to X))\cdot i_{I*}\mathcal{O}_{D_I}\bigr)=(-1-y)^{|I|}\chi_y(D_I^\circ).
    \]
\end{corollary}
\begin{proof}
    Using the identity
    \[
        i_{I*}i_I^*(a)=a\cdot i_{I*}\mathcal{O}_{D_I}
    \]
    in \(K(X)\), we get
    \[
        \chi\bigl(\mathbb{D}_X(\MC_y(U\to X))\cdot i_{I*}\mathcal{O}_{D_I}\bigr)
        =\chi\bigl(i_{I*}i_I^*\mathbb{D}_X(\MC_y(U\to X))\bigr)
        =\chi\bigl(i_I^*\mathbb{D}_X(\MC_y(U\to X))\bigr).
    \]
    Since \(X\) is smooth of dimension \(n\), we have
    \[
        \mathbb{D}_X(\MC_y(U\to X))=(-1)^n\,\MC_y(U\to X)^\vee\otimes \omega_X
    \]
    in \(K(X)[y]\). Therefore
    \[
        \chi\bigl(i_I^*\mathbb{D}_X(\MC_y(U\to X))\bigr)
        =(-1)^n\chi\bigl(i_I^*\MC_y(U\to X)^\vee\otimes i_I^*\omega_X\bigr).
    \]
    Since $D_I$ is smooth, every K-class on $D_I$ is represented by a perfect complex and Serre duality holds in $K(D_I)$. Applying it on $D_I$, of dimension $n-|I|$, we obtain
    \[
        \chi\bigl(i_I^*\mathbb{D}_X(\MC_y(U\to X))\bigr)
        =(-1)^n(-1)^{n-|I|}\chi\bigl(i_I^*\MC_y(U\to X)\otimes i_I^*\omega_X^{-1}\otimes \omega_{D_I}\bigr).
    \]
    By the adjunction formula, \(i_I^*\omega_X^{-1}\otimes \omega_{D_I}\cong \det N_{D_I/X}\). Hence
    \[
        \chi\bigl(\mathbb{D}_X(\MC_y(U\to X))\cdot i_{I*}\mathcal{O}_{D_I}\bigr)
        =(-1)^{|I|}\chi\bigl(i_I^*\MC_y(U\to X)\otimes \det N_{D_I/X}\bigr).
    \]
    Applying the projection formula in the opposite direction,
    \[
        \chi\bigl(\mathbb{D}_X(\MC_y(U\to X))\cdot i_{I*}\mathcal{O}_{D_I}\bigr)
        =(-1)^{|I|}\chi\bigl(\MC_y(U\to X)\cdot i_{I*}\det N_{D_I/X}\bigr).
    \]
    By \cref{prop:MCy-SNC-pairing}, the right-hand side equals $(-1)^{|I|}(1+y)^{|I|}\chi_y(D_I^\circ)=(-1-y)^{|I|}\chi_y(D_I^\circ)$, as desired.
\end{proof}

\subsection{Motivic Chern classes of realizable matroids}
We now apply \cref{cor:GW-for-MCy} to the wonderful compactification \(W_L\). For each nonempty proper flat \(F\) of \(\mathsf{M}\), let \(D_F\subset W_L\) be the corresponding irreducible boundary divisor. If
\[
    \mathcal{F}=\emptyset\subsetneq F_1\subsetneq \cdots \subsetneq F_k\subsetneq E
\]
is a flag of flats, we write
\[
    D_{\mathcal{F}}=D_{F_1}\cap \cdots \cap D_{F_k}
    \qquad\text{and}\qquad
    D_{\mathcal{F}}^\circ
    =D_{\mathcal{F}}\setminus \bigcup_{G\notin\{F_1,\ldots,F_k\}}D_G.
\]
Then \(W_L\setminus U_L=\bigcup_F D_F\) is an SNC divisor, and \(D_{\mathcal{F}}=W_L\cap V(\sigma_{\mathcal{F}})\), so this notation is compatible with \cref{prop:MCy-SNC-pairing}.

We use the successive-minor notation introduced in the introduction. Here and later in the paper, \(\mathsf{M}|F\) denotes the restriction of \(\mathsf{M}\) to the flat \(F\), and \(\mathsf{M}/F\) denotes the contraction.

\motivicChernWeight*
\begin{proof}
    Let \(\mathcal{F}=\emptyset\subsetneq F_1\subsetneq \cdots \subsetneq F_k\subsetneq E\). Applying \cref{cor:GW-for-MCy} to \(X=W_L\), \(U=U_L\), and the boundary components \(D_{F_1},\ldots,D_{F_k}\),
    \[
        g_{\mathsf{M}}^{\mathbb{D}}(\mathcal{F})=(-1-y)^k\,\chi_y(D_{\mathcal{F}}^\circ).
    \]
    The wonderful compactification gives a canonical isomorphism
    \[
        D_{\mathcal{F}}^\circ \;\cong\; U_{\mathsf{M}|F_1}\times U_{\mathsf{M}|F_2/F_1}\times\cdots\times U_{\mathsf{M}/F_k},
    \]
    where each factor \(U_{\mathsf{N}}\) is the complement of a realizing hyperplane arrangement for \(\mathsf{N}\). Since \(\chi_y\) is multiplicative under products, and for each minor \(\mathsf{N}\) we have \(\chi_y(U_{\mathsf{N}})=\overline{\chi}_{\mathsf{N}}(-y)\) by \cite[Corollary~2.2]{Aluffi2013}, we obtain
    \[
        \chi_y(D_{\mathcal{F}}^\circ)=\overline{\chi}_{\mathsf{M}}(-y)[\mathcal{F}].
    \]
    The stated formula follows from \(\chi_{\mathsf{N}}(t)=(t-1)\overline{\chi}_{\mathsf{N}}(t)\) applied to each of the \(k+1\) factors in \(\chi_{\mathsf{M}}(-y)[\mathcal{F}]\).
\end{proof}

\begin{corollary}\label{cor:independence-of-L}
    The function \(g_{\mathsf{M}}^{\mathbb{D}}\colon \Sigma_{\mathsf{M}}\to \Z[y]\) depends only on the matroid \(\mathsf{M}\), and not on the choice of realization \(L\).
\end{corollary}

\begin{remark}\label{rmk:independence-of-L}
    If one identifies \(K(W_L)\) with the matroid \(K\)-ring \(K(\mathsf{M})\) via Grothendieck weights as discussed in the previous section, then the preceding corollary shows that \(\mathbb{D}_{W_L}(\MC_y(\mathsf{M}))\), and hence also \(\MC_y(\mathsf{M})\), is independent of the choice of realization \(L\).
\end{remark}

\subsection{Motivic Chern classes of general matroids}
Next, we define \(\MC_y(\mathsf{M})\) in the non-realizable case. We use \cref{prop:GW-dual-MCy-matroid} to define the motivic Chern class of a general matroid. Our goal is to prove the following theorem.

\motivicChernGeneral*

We first need two combinatorial lemmas. For a matroid $\mathsf{N}$, we write $L(\mathsf{N})$ for its lattice of flats, and $\beta(\mathsf{N})$ for Crapo's beta invariant \((-1)^{\rk \mathsf{N}-1}\overline{\chi}_{\mathsf{N}}(1)\).

\begin{lemma}\label{lem:pointed-convolution}
    For every loopless matroid \(\mathsf{N}\) on $E=[n]$ of rank $r$ and every \(i\in E\),
    \[
        \overline{\chi}_{\mathsf{N}}(t)
        =
        \sum_{\substack{F\in L(\mathsf{N})\\ i\in F}}
        (-1)^{\rk(F)-1}\beta(\mathsf{N}|F)\,\chi_{\mathsf{N}/F}(t).
    \]
\end{lemma}
\begin{proof}
    It suffices to expand the definitions. By definition of the characteristic polynomial, see, for example, \cite[Section~7.3]{Oxley2011},
    \[
        \chi_{\mathsf{N}/F}(t)=\sum_{F\subseteq G}\mu(F,G)t^{r-\rk(G)},
    \]
    and by applying Weisner's theorem to \(\mathsf{N}|F\), see, for example, \cite[Section~7.4]{Oxley2011},
    \[
        (-1)^{\rk(F)-1}\beta(\mathsf{N}|F)=\overline{\chi}_{\mathsf{N}|F}(1)=\sum_{\substack{H\in L(\mathsf{N})\\H\subseteq F,\ i\not\in H}}\mu(\emptyset,H).
    \]
    Therefore the right-hand side is
    \[
        \sum_{\substack{H\subseteq F\subseteq G\\i\not\in H,\ i\in F}}\mu(\emptyset,H)\mu(F,G)t^{r-\rk(G)}.
    \]
    Rearranging the summation gives
    \[
        \sum_{\substack{H\subseteq G\\i\not\in H}}\mu(\emptyset,H)t^{r-\rk(G)}
        \left(\sum_{\substack{F:\,\overline{H\cup\{i\}}\subseteq F\subseteq G}}\mu(F,G)\right),
    \]
    where \(\overline{H\cup\{i\}}\) denotes the closure of \(H\cup\{i\}\). By the defining property of the M\"obius function on a poset, see, for example, \cite[Section~3.7]{StanleyEC1}, the inner sum is zero unless \(G=\overline{H\cup\{i\}}\), in which case it is equal to \(1\). Hence the above summation simplifies to
    \[
        \sum_{i\not\in H}\mu(\emptyset,H)t^{r-\rk(\overline{H\cup\{i\}})}.
    \]
    Since \(H\) is a flat and \(i\not\in H\), we have \(\rk(\overline{H\cup\{i\}})=\rk(H)+1\). Thus the previous expression is equal to
    \[
        \sum_{i\not\in H}\mu(\emptyset,H)t^{r-\rk(H)-1},
    \]
    which is exactly the pointed expression for \(\overline{\chi}_{\mathsf{N}}(t)\) given by Weisner's theorem.
\end{proof}

\begin{lemma}\label{lem:Psi-formula}
    For every loopless matroid \(\mathsf{N}\) on $E=[n]$ of rank $r$ and every \(i\in E\),
    \[
        \sum_{\substack{\mathcal{F}\\i\in H,\forall H\in \mathcal{F}}}(-1)^{\ell(\mathcal{F})} \chi_{\mathsf{N}}(t)[\mathcal{F}]=(-1)^{r-1}\beta(\mathsf{N})(t-1).
    \]
    Note that we allow the empty flag in the summation.
\end{lemma}
\begin{proof}
    We prove the lemma by induction on the rank \(r\). If \(r=1\), then \(\mathsf{N}\) has no nonempty proper flats, so the only flag contributing to the summation is the empty flag. Hence the left-hand side is
    \[
        \chi_{\mathsf{N}}(t)=t-1.
    \]
    On the other hand, \(\overline{\chi}_{\mathsf{N}}(t)=1\), so \(\beta(\mathsf{N})=1\). Therefore the desired identity holds in this case.

    Assume now \(r\ge 2\). Separating the empty flag from the summation, we obtain
    \[
        \sum_{\substack{\mathcal{F}\\i\in H,\forall H\in \mathcal{F}}}(-1)^{\ell(\mathcal{F})} \chi_{\mathsf{N}}(t)[\mathcal{F}]
        =
        \chi_{\mathsf{N}}(t)
        +
        \sum_{\substack{\mathcal{F}\neq\emptyset\\i\in H,\forall H\in \mathcal{F}}}(-1)^{\ell(\mathcal{F})} \chi_{\mathsf{N}}(t)[\mathcal{F}].
    \]

    To further simplify the second summation, let \(F_k\) be the largest proper flat in \(\mathcal{F}\). Since \(i\in H\) for all \(H\in \mathcal{F}\), in particular \(i\in F_k\). Therefore, we separate out \(F_k\) in the summation:
    \begin{align*}
        & \sum_{\substack{\mathcal{F}\neq\emptyset\\i\in H,\forall H\in \mathcal{F}}}(-1)^{\ell(\mathcal{F})} \chi_{\mathsf{N}}(t)[\mathcal{F}] \\
        ={} & \sum_{\substack{\mathcal{F}:F_1\subsetneq\cdots\subsetneq F_k,k>0 \\i\in H,\forall H\in \mathcal{F}}}(-1)^k \chi_{\mathsf{N}|F_1}(t)\cdots \chi_{\mathsf{N}|F_k/F_{k-1}}(t)\chi_{\mathsf{N}/F_k}(t)\\
        ={} & -\sum_{\substack{F\in L(\mathsf{N})\\i\in F,\ F\neq E}}\left(\sum_{\substack{\mathcal{F}'\\i\in H,\forall H\in \mathcal{F}'}}\!\!(-1)^{\ell(\mathcal{F}')} \chi_{\mathsf{N}|F}(t)[\mathcal{F}']\right)\chi_{\mathsf{N}/F}(t).
    \end{align*}
    Here \(\mathcal{F}'\) ranges over all flags of nonempty proper flats of \(\mathsf{N}|F\), including the empty flag. Applying the induction hypothesis to \(\mathsf{N}|F\), we obtain
    \[
        \sum_{\substack{\mathcal{F}\neq\emptyset\\i\in H,\forall H\in \mathcal{F}}}(-1)^{\ell(\mathcal{F})} \chi_{\mathsf{N}}(t)[\mathcal{F}]
        =
        -(t-1)\sum_{\substack{F\in L(\mathsf{N})\\i\in F,\ F\neq E}}(-1)^{\rk(F)-1}\beta(\mathsf{N}|F)\chi_{\mathsf{N}/F}(t).
    \]

    By \cref{lem:pointed-convolution},
    \[
        \overline{\chi}_{\mathsf{N}}(t)
        =
        \sum_{\substack{F\in L(\mathsf{N})\\ i\in F}}
        (-1)^{\rk(F)-1}\beta(\mathsf{N}|F)\,\chi_{\mathsf{N}/F}(t).
    \]
    Separating the term \(F=E\), we get
    \[
        \sum_{\substack{F\in L(\mathsf{N})\\i\in F,\ F\neq E}}(-1)^{\rk(F)-1}\beta(\mathsf{N}|F)\chi_{\mathsf{N}/F}(t)
        =
        \overline{\chi}_{\mathsf{N}}(t)-(-1)^{r-1}\beta(\mathsf{N}),
    \]
    since \(\chi_{\mathsf{N}/E}(t)=1\). Therefore,
    \begin{align*}
        \sum_{\substack{\mathcal{F}\\i\in H,\forall H\in \mathcal{F}}}(-1)^{\ell(\mathcal{F})} \chi_{\mathsf{N}}(t)[\mathcal{F}]
        &=
        \chi_{\mathsf{N}}(t)-(t-1)\bigl(\overline{\chi}_{\mathsf{N}}(t)-(-1)^{r-1}\beta(\mathsf{N})\bigr) \\
        &=
        (-1)^{r-1}\beta(\mathsf{N})(t-1),
    \end{align*}
    where the last equality follows from \(\chi_{\mathsf{N}}(t)=(t-1)\overline{\chi}_{\mathsf{N}}(t)\).
\end{proof}

\begin{proof}[Proof of \cref{thm:motivic-Chern-non-realizable}]
    Let \(t=-y\). We prove the theorem by verifying the $K$-balancing condition of \cref{prop:K-balancing-for-matroid}. Since \(t-1=-1-y\) is a nonzero element of the integral domain \(\Z[y]\), a \(\Z[y]\)-valued function satisfies the balancing condition if and only if its product with \(t-1\) does. It therefore suffices to check the balancing condition for \((t-1)g_{\mathsf{M}}^{\mathbb{D}}=\chi_{\mathsf{M}}(t)[\cdot]\).

    We first prove the balancing condition for \(\mathcal{G}=\emptyset\). In this case, we need to show that for every loopless matroid \(\mathsf{M}\) on \(E=[n]\) and every pair \(i,j\in E\),
    \begin{equation}\label{eq:Aij-symmetry}
        A_{i,j}(\mathsf{M})=A_{j,i}(\mathsf{M}),
    \end{equation}
    where we set
    \[
        A_{i,j}(\mathsf{M})
        \coloneq
        \sum_{\substack{i\in H,\;j\notin H\\ \forall H\in\mathcal{F}}}
        (-1)^{\ell(\mathcal{F})}\chi_{\mathsf{M}}(t)[\mathcal{F}].
    \]
    Here the sum ranges over nonempty flags \(\mathcal{F}\) of nonempty proper flats of \(\mathsf{M}\).
    To further simplify the summation, we note that the condition \(j\not\in H\) for all \(H\in \mathcal{F}\) is equivalent to saying $j\not\in F_k$, where $F_k$ is the largest proper flat in $\mathcal{F}$. Therefore, we separate out $F_k$ in the summation:
    \begin{align*}
        A_{i,j}(\mathsf{M}) & =\sum_{\substack{\mathcal{F}:F_1\subsetneq\cdots\subsetneq F_k,k>0 \\i\in F_1,j\not\in F_k}}(-1)^k \chi_{\mathsf{M}|F_1}(t)\cdots \chi_{\mathsf{M}|F_{k}/F_{k-1}}(t)\chi_{\mathsf{M}/F_k}(t)\\
                            & =\sum_{\substack{F_k\in L(\mathsf{M})\\i\in F_k,\ j\not\in F_k}}\left(\sum_{\substack{\mathcal{F}':F_1\subsetneq\cdots\subsetneq F_{k-1}\\i\in H,\forall H\in \mathcal{F}'}}(-1)^k \chi_{\mathsf{M}|F_1}(t)\cdots \chi_{\mathsf{M}|F_k/F_{k-1}}(t)\right)\chi_{\mathsf{M}/F_k}(t)\\
                            & =\sum_{\substack{F_k\in L(\mathsf{M})\\i\in F_k,\ j\not\in F_k}}\sum_{\substack{\mathcal{F}'\\i\in H,\forall H\in \mathcal{F}'}}-(-1)^{\ell(\mathcal{F}')} \chi_{\mathsf{M}|F_k}(t)[\mathcal{F}']\chi_{\mathsf{M}/F_k}(t).
    \end{align*}

    Applying \cref{lem:Psi-formula} to \(\mathsf{N}=\mathsf{M}|F_k\), and noting that \(\mathsf{M}|F_j/F_{j-1}=(\mathsf{M}|F_k)|F_j/F_{j-1}\), we obtain
    \[
        A_{i,j}(\mathsf{M})=-(t-1)\sum_{\substack{F\in L(\mathsf{M})\\i\in F,\ j\not\in F}}(-1)^{\rk(F)-1}\beta(\mathsf{M}|F)\chi_{\mathsf{M}/F}(t).
    \]
    Here \(F=F_k\) ranges over all flats of \(\mathsf{M}\) containing \(i\) and not containing \(j\).
    
    Furthermore, 
    \begin{align*}
        & \sum_{\substack{F\in L(\mathsf{M})\\i\in F,\ j\not\in F}}(-1)^{\rk(F)-1}\beta(\mathsf{M}|F)\chi_{\mathsf{M}/F}(t) \\
        = & \sum_{\substack{F\in L(\mathsf{M})\\i\in F}}(-1)^{\rk(F)-1}\beta(\mathsf{M}|F)\chi_{\mathsf{M}/F}(t) 
        -
        \sum_{\substack{F\in L(\mathsf{M})\\i,j\in F}}(-1)^{\rk(F)-1}\beta(\mathsf{M}|F)\chi_{\mathsf{M}/F}(t).
    \end{align*}
    By \cref{lem:pointed-convolution}, the first summation is
    \(
        \overline{\chi}_{\mathsf{M}}(t),
    \)
    hence independent of \(i\). Therefore the right-hand side is symmetric in \(i\) and \(j\), and we obtain \(A_{i,j}=A_{j,i}\).
    
    For a general \(\{i,j\}\)-neutral flag \(\mathcal{G}\), write
    \[
        \emptyset=G_0\subsetneq G_1\subsetneq\cdots\subsetneq G_m\subsetneq G_{m+1}=E.
    \]
    Since \(\mathcal{G}\) is \(\{i,j\}\)-neutral, there is a unique index \(s\) such that \(i,j\not\in G_s\) and \(i,j\in G_{s+1}\). Set
    \[
        \mathsf{N}= (\mathsf{M}|G_{s+1})/G_s.
    \]
    Then the interval \([G_s,G_{s+1}]\subseteq L(\mathsf{M})\) is naturally identified with \(L(\mathsf{N})\). Under this identification, every refinement \(\mathcal{F}\sref \mathcal{G}\) contributing to the left-hand side of the balancing condition in \cref{prop:K-balancing-for-matroid} is obtained by inserting a nonempty flag \(\mathcal{H}\) of nonempty proper flats of \(\mathsf{N}\) such that \(i\in H\) and \(j\not\in H\) for every \(H\in \mathcal{H}\). Moreover,
    \[
        (-1)^{\ell(\mathcal{F})}\chi_{\mathsf{M}}(t)[\mathcal{F}]
        =
        C_{\mathcal{G}}\,(-1)^{\ell(\mathcal{H})}\chi_{\mathsf{N}}(t)[\mathcal{H}],
    \]
    where \(C_{\mathcal{G}}\) is the product of the factors \(\chi_{(\mathsf{M}|G_{a+1})/G_a}(t)\) coming from the intervals of \(\mathcal{G}\) other than \([G_s,G_{s+1}]\), multiplied by \((-1)^{\ell(\mathcal{G})}\). Hence the left-hand side of the balancing condition is equal to \(C_{\mathcal{G}}A_{i,j}(\mathsf{N})\), and similarly the right-hand side is equal to \(C_{\mathcal{G}}A_{j,i}(\mathsf{N})\). By the empty-flag case proved above, these two quantities are equal. Therefore \(\chi_{\mathsf{M}}(-y)[\cdot]\), and hence also \(g_{\mathsf{M}}^{\mathbb{D}}\), satisfies the balancing condition of \cref{prop:K-balancing-for-matroid}.
\end{proof}

We are in a position to define the motivic Chern class of a loopless matroid $\mathsf{M}$. Recall that in \cite[Theorem~6.2]{LARSON2024109554}, the Serre duality operator is also defined in \(K(\mathsf{M})\) for a general loopless \(\mathsf{M}\): \(\mathbb{D}(\zeta)=\omega_M \zeta^\vee\). Since \(\Phi:K(\mathsf{M})\xrightarrow{\sim}\GW(\Sigma_{\mathsf{M}})\) is an isomorphism and \(\mathbb{D}\) is invertible on \(K(\mathsf{M})\), there is a unique element whose Grothendieck weight after applying \(\mathbb{D}\) is \(g_{\mathsf{M}}^{\mathbb{D}}\).

\begin{definition}\label{def:motivic-Chern-general}
    For a loopless matroid \(\mathsf{M}\), the motivic Chern class is the unique element \(\MC_y(\mathsf{M})\in K(\mathsf{M})\) such that the Grothendieck weight of \(\mathbb{D}(\MC_y(\mathsf{M}))\) is \(g_{\mathsf{M}}^{\mathbb{D}}\).
\end{definition}

By \cref{prop:GW-dual-MCy-matroid}, this definition coincides with \cref{def:motivic-Chern-realizable} when $\mathsf{M}$ is realizable.

\subsection{Specialization to CSM classes}
In this subsection, we show that the specialization of \(\MC_y(\mathsf{M})\) at \(y=-1\) recovers the Chern--Schwartz--MacPherson classes of matroids. Since the varieties appearing here are smooth and proper, we view the Hirzebruch class transformation
\[
    T_{y*}=td_{(1+y)}\circ \MC_y
\]
as taking values in the rational Chow group \(A_*(-)_\Q[y]\). For \(\beta\in K(X)\), write
\[
    \ch_{(1+y)}(\beta)\coloneq \sum_{p\ge 0}(1+y)^p \ch^p(\beta)\in A^*(X)_\Q[y].
\]
For \(\alpha\in K(X)[y]\), write
\[
    td_*(\alpha)=\sum_{p\ge 0} td_p(\alpha),
\]
where \(td_p(\alpha)\) denotes the \(A_p(X)\)-component of \(td_*(\alpha)\), and set
\[
    td_{(1+y)}(\alpha)\coloneq \sum_{p\ge 0}(1+y)^{-p}td_p(\alpha).
\]
Thus \(T_{y*}(U\to X)=td_{(1+y)}(\MC_y(U\to X))\).

\begin{proposition}\label{prop:CSM-SNC-pairing}
    Let \(X\) be a smooth proper variety, \(D=\bigcup_{i=1}^m D_i\) an SNC divisor, and \(U=X\setminus D\). For each \(I\subseteq \{1,\ldots,m\}\), set \(D_I=\bigcap_{i\in I}D_i\), let \(i_I:D_I\hookrightarrow X\) be the closed embedding, and let \(D_I^\circ=D_I\setminus \bigcup_{j\notin I}D_j\). Then
    \[
        \int_X ch_{(1+y)}\bigl(i_{I*}\det N_{D_I/X}\bigr)\cap T_{y*}(U\to X)
        =(1+y)^{|I|}\chi_y(D_I^\circ).
    \]
    In particular,
    \[
        \int_X [D_I]\cap c_{\mathrm{SM}}(1_U)=\chi(D_I^\circ).
    \]
\end{proposition}
\begin{proof}
    By Hirzebruch--Riemann--Roch,
    \[
        \chi(\alpha\cdot\beta)=\int_X \ch(\beta)\cap td_*(\alpha)
    \]
    for \(\alpha\in K(X)[y]\) and \(\beta\in K(X)\). Expanding by dimension gives
    \[
        \chi(\alpha\cdot\beta)=\sum_{p\ge 0}\int_X \ch^p(\beta)\cap td_p(\alpha).
    \]
    Therefore
    \begin{align*}
        \chi(\alpha\cdot\beta)
         &=\sum_{p\ge 0}\int_X (1+y)^p \ch^p(\beta)\cap (1+y)^{-p}td_p(\alpha) \\
         &=\int_X ch_{(1+y)}(\beta)\cap td_{(1+y)}(\alpha).
    \end{align*}
    Applying this to \(\alpha=\MC_y(U\to X)\) and \(\beta=i_{I*}\det N_{D_I/X}\), and then using \cref{prop:MCy-SNC-pairing}, we obtain
    \[
        \int_X ch_{(1+y)}\bigl(i_{I*}\det N_{D_I/X}\bigr)\cap T_{y*}(U\to X)
        =\chi\bigl(\MC_y(U\to X)\cdot i_{I*}\det N_{D_I/X}\bigr)
        =(1+y)^{|I|}\chi_y(D_I^\circ).
    \]

    To specialize at \(y=-1\), we apply Grothendieck--Riemann--Roch to the regular embedding \(i_I\):
    \[
        ch\bigl(i_{I*}\det N_{D_I/X}\bigr)
        =i_{I*}\bigl(\ch(\det N_{D_I/X})\,td(N_{D_I/X})^{-1}\bigr).
    \]
    Hence \(\ch^p(i_{I*}\det N_{D_I/X})=0\) for \(p<|I|\), and the codimension-\(|I|\) component is
    \[
        i_{I*}\bigl(\ch^0(\det N_{D_I/X})\,td^0(N_{D_I/X})^{-1}\bigr)=i_{I*}(1)=[D_I].
    \]
    Thus \((1+y)^{-|I|}ch_{(1+y)}\bigl(i_{I*}\det N_{D_I/X}\bigr)=[D_I]+(1+y)\gamma_I(y)\) for some \(\gamma_I(y)\in A^*(X)_\Q[y]\). Dividing the displayed integral formula above by \((1+y)^{|I|}\) and substituting,
    \[
        \int_X \bigl([D_I]+(1+y)\gamma_I(y)\bigr)\cap T_{y*}(U\to X)=\chi_y(D_I^\circ).
    \]
    Setting \(y=-1\) and using \(T_{-1,*}=c_{\mathrm{SM}}\otimes\Q\) and \(\chi_{-1}=\chi\) (topological Euler characteristic) from \cite{BrasseletSchurmannYokura2010}, we obtain
    \[
        \int_X [D_I]\cap c_{\mathrm{SM}}(1_U)=\chi(D_I^\circ).\qedhere
    \]
\end{proof}

\csmSpecialization*
\begin{proof}
    Let \(L\) realize \(\mathsf{M}\), and let \(j:W_L\hookrightarrow X_{[n]}\) be the natural closed embedding into the complete permutohedral variety. Applying \cref{prop:CSM-SNC-pairing} to \(X=W_L\), \(U=U_L\), and the boundary components \(D_{F_1},\ldots,D_{F_k}\), we obtain
    \[
        \deg\!\bigl(c_{\mathrm{SM}}(1_{U_L})\cdot [D_{\mathcal{F}}]\bigr)=\chi(D_{\mathcal{F}}^\circ).
    \]
    Since \(D_{\mathcal{F}}=W_L\cap V(\sigma_{\mathcal{F}})\), we have \(j^*[V(\sigma_{\mathcal{F}})]=[D_{\mathcal{F}}]\). The projection formula then gives
    \[
        \deg\!\bigl(j_*c_{\mathrm{SM}}(1_{U_L})\cdot [V(\sigma_{\mathcal{F}})]\bigr)=\chi(D_{\mathcal{F}}^\circ).
    \]
    As before, the wonderful compactification gives
    \[
        D_{\mathcal{F}}^\circ \cong U_{\mathsf{M}|F_1}\times U_{\mathsf{M}|F_2/F_1}\times\cdots\times U_{\mathsf{M}/F_k}.
    \]
    Therefore
    \[
        \chi(D_{\mathcal{F}}^\circ)=\overline{\chi}_{\mathsf{M}}(1)[\mathcal{F}],
    \]
    because \(\chi(U_{\mathsf{N}})=\overline{\chi}_{\mathsf{N}}(1)\) for every minor \(\mathsf{N}\), which follows from \cite[Corollary~2.2]{Aluffi2013} at \(y=-1\). Finally, for every minor \(\mathsf{N}\),
    \[
        \overline{\chi}_{\mathsf{N}}(1)=(-1)^{\rk(\mathsf{N})-1}\beta(\mathsf{N}),
    \]
    and applying this to each factor in \(\overline{\chi}_{\mathsf{M}}(1)[\mathcal{F}]\) gives the beta-invariant formula. By the identification of Chow classes on \(X_{[n]}\) with Minkowski weights on \(\Sigma_{[n]}\), the degree computed above is \(\operatorname{csm}_k(\mathsf{M})(\mathcal{F})\), proving the formula.
\end{proof}

\section{Tautological bundles on the permutohedral variety}

In this section, we compute the Grothendieck weights of the tautological bundles associated with a matroid and derive a combinatorial identity for the Tutte polynomial.

We follow the notation and conventions of \cite{BergetEurSpinkTseng2023}. For a loopless matroid $\mathsf{M}$ of rank $r$ on $E = [n]$, Berget--Eur--Spink--Tseng define two $K$-classes on $X = X_{[n]}$: the \emph{tautological subbundle} $[\mathcal{S}_{\mathsf{M}}]$ of rank $r$ and the \emph{tautological quotient bundle} $[\mathcal{Q}_{\mathsf{M}}]$ of rank $n - r$. When $\mathsf{M}$ is realized by a linear subspace $L \subset \mathbb{C}^E$, these are the equivariant subbundle and quotient bundle of the trivial bundle $\underline{\mathbb{C}}^E$ determined by $L$.
 
Note that we are working in $K(X_{[n]})$, so the cones are labeled by flags of subsets rather than flags of flats.

We continue to use the successive-minor notation introduced in the introduction.

\tautologicalWeight*
\begin{proof}
    We first compute the total Euler characteristic. By \cite[Theorem~10.1]{BergetEurSpinkTseng2023}, there is a ring isomorphism $\zeta : K(X) \xrightarrow{\sim} A^\bullet(X)$ satisfying $\chi([\mathcal{E}]) = \deg_\alpha(\zeta([\mathcal{E}]))$, where $\alpha \in A^1(X)$ is the hyperplane class. Since $\zeta$ is a ring isomorphism and $\mathcal{S}_{\mathsf{M}}, \mathcal{Q}_{\mathsf{M}}$ have simple Chern roots in the sense of \cite[Proposition~10.5]{BergetEurSpinkTseng2023}, that proposition gives
    \[
        \zeta\!\left(\lambda_u(\mathcal{S}_{\mathsf{M}}^\vee)\right) = (u+1)^r\, c\!\left(\mathcal{S}_{\mathsf{M}}^\vee,\tfrac{u}{u+1}\right)
    \]
    and
    \[
        \zeta\!\left(\lambda_v(\mathcal{Q}_{\mathsf{M}}^\vee)\right) = (v+1)^{n-r}\, c\!\left(\mathcal{Q}_{\mathsf{M}}^\vee,\tfrac{v}{v+1}\right) = (v+1)^{n-r}\, c\!\left(\mathcal{Q}_{\mathsf{M}},-\tfrac{v}{v+1}\right).
    \]
    Therefore
    \[
        \chi\!\left(\lambda_u(\mathcal{S}_{\mathsf{M}}^\vee)\lambda_v(\mathcal{Q}_{\mathsf{M}}^\vee)\right)
        = (u+1)^r(v+1)^{n-r}
          \sum_{i+k+\ell = n-1}
          \left(\int_X \alpha^i c_k(\mathcal{S}_{\mathsf{M}}^\vee)\,c_\ell(\mathcal{Q}_{\mathsf{M}})\right)
          \left(\tfrac{u}{u+1}\right)^k\!\left(-\tfrac{v}{v+1}\right)^\ell.
    \]
    By \cite[Theorem~A]{BergetEurSpinkTseng2023}, this integral equals
    \begin{align*}
        (u+1)^r(v+1)^{n-r}\,t_{\mathsf{M}}\!\left(1,0,\tfrac{u}{u+1},-\tfrac{v}{v+1}\right)
        &= (u+1)^r(v+1)^{n-r}\left(\tfrac{u}{u+1}\right)^r\!\left(1-\tfrac{v}{v+1}\right)^{n-r}
           T_{\mathsf{M}}\!\left(\tfrac{u+1}{u},\,1+v\right) \\
        &= u^r T_{\mathsf{M}}\!\left(1+\tfrac{1}{u},\,1+v\right).
    \end{align*}
    For a general flag $\mathcal{F}$, the Grothendieck weight is $\chi(\lambda_u(\mathcal{S}_{\mathsf{M}}^\vee)\lambda_v(\mathcal{Q}_{\mathsf{M}}^\vee) \cdot x_{\mathcal{F}})$. By the minor decomposition property \cite[Propositions~5.2 and~5.3]{BergetEurSpinkTseng2023}, the tautological bundles restricted to $V(\sigma_{\mathcal{F}})$ decompose as products over the successive minors of $\mathcal{F}$, so applying the same computation to each minor gives
    \[
        u^r T_{\mathsf{M}}\!\left(1+\tfrac{1}{u},\,1+v\right)[\mathcal{F}].
    \]
\end{proof}

Expanding the Tutte polynomial as $u^r T_{\mathsf{M}}(1+1/u,\,1+v) = \sum_{A \subseteq E} u^{r(A)} v^{|A|-r(A)}$ and specializing $v=0$ and $u=0$, respectively, yields the following.

\begin{corollary}\label{cor:gw-taut-specializations}
    Let $\mathsf{M}$ be a loopless matroid on $E = [n]$.
    \begin{enumerate}
        \item Write $I_{\mathsf{M}}(u) = \sum_{I \text{ independent}} u^{|I|}$ for the $f$-polynomial of the independence complex of $\mathsf{M}$. The Grothendieck weight of $\lambda_u(\mathcal{S}_{\mathsf{M}}^\vee)$ is
        \[
            \mathcal{F} \mapsto I_{\mathsf{M}}(u)[\mathcal{F}].
        \]
        \item Write $\operatorname{loop}(\mathsf{N})$ for the number of loops of a matroid $\mathsf{N}$. The Grothendieck weight of $\lambda_v(\mathcal{Q}_{\mathsf{M}}^\vee)$ is
        \[
            \mathcal{F} \mapsto (1+v)^{\operatorname{loop}(\mathsf{M})}[\mathcal{F}].
        \]
        Note that even though $\mathsf{M}$ is loopless, the successive minors of $\mathcal{F}$ may have loops (via contraction), so this formula is nontrivial on non-flat flags.
    \end{enumerate}
\end{corollary}

Combining \cref{prop:gw-tautological,cor:gw-taut-specializations} with the product rule (\cref{cor:product-rule-general}), we obtain the following Tutte polynomial identity.

\tutteProduct*

\printbibliography

\end{document}